\documentclass[preprint]{imsart}

\RequirePackage[OT1]{fontenc}
\RequirePackage{amsthm,amsmath}
\RequirePackage[numbers]{natbib}
\RequirePackage[colorlinks,citecolor=blue,urlcolor=blue]{hyperref}
\usepackage{graphicx}
\usepackage[english]{babel}
\usepackage[T1]{fontenc}
\usepackage{amsmath}
\usepackage{amssymb}
\usepackage{graphicx}
\usepackage{amscd}

\newtheorem{definition}{Definition}[section]
\newtheorem{theorem}[definition]{Theorem}
\newtheorem{lemma}[definition]{Lemma}

\newtheorem{remark}{Remark}

\usepackage{color}

\def\a {\alpha}

\def\o {\omega}
\def\O {\Omega}
\def \veps {\varepsilon}
\def \vphi {\varphi}

\def\CP {\mathcal{P}}

\def\CI {\mathcal{I}}
\def\CF {\mathcal{F}}
\def\CM {\mathcal{M}}

\def\CB {\mathcal{B}}
\def\CS {\mathcal{S}}
\def\CR {\mathcal{R}}

\def\CA { \mathcal{A}}

\def\CR {\mathcal{R}}

\def\E {\mathbb{E}}

\def\P {\mathbb{P}}
\def\R {\mathbb{R}}
\def\N {\mathbb{N}}
\def\I {\mathbb{I}}
\def\Q {\mathbb{Q}}

\def\ti {\tilde{i}}
\def\tX {\tilde{X}}
\def\tnu {\tilde{\nu}}
\def\tbe {\tilde{\beta}}
\def\tlm {\tilde{\lambda}}

\def\tu {\tilde{u}}
\def\ta {\tilde{a}}
\def\tLm {\tilde{\Lambda}}

\def\tL{\widetilde{L}}
\def\tR{\widetilde{R}}
\def\tM{\tilde{M}}

\def\cO {\widehat{\Omega}}
\def\cF {\widehat{\CF}}
\def\cP {\widehat{\CP}}
\def\cW {\widehat{W}}

\def\cnun {\hat{\nu}^{(n)}}
\def\cLn {\hat{L}^{(n)}}
\def\cRn {\hat{R}^{(n)}}

\def\cin {\hat{i}^{(n)}}
\def\cben {\hat{\beta}^{(n)}}
\def\clmn {\hat{\lambda}^{(n)}}
\def\cLmn {\widehat{\Lambda}^{(n)}}
\def\cun {\hat{u}^{(n)}}

\def\cMn {\hat{M}^{(n)}}
\def \cLn{\hat{L}^{(n)}}
\def \cRn{\hat{R}^{(n)}}

\def\cbe {\hat{\beta}}
\def\cnu {\hat{\nu}}

\def\ci {\hat{i}}
\def\ca {\hat{a}}
\def\co {\hat{\omega}}
\def\cL {\hat{L}}
\def\cR {\hat{R}}
\def\cY {\hat{Y}}

\def\Minf {M_{\infty}^{(\a)}}
\begin{document}
\begin{frontmatter}
\title{Characterization of weak convergence of probability-valued solutions of general one-dimensional kinetic equations}
\runtitle{Convergence of solutions of kinetic equations}


\begin{aug}
\author{\fnms{Eleonora} \snm{Perversi}\thanksref{t1}\ead[label=e1]{eleonora.perversi@unipv.it}}
\and
\author{\fnms{Eugenio} \snm{Regazzini}\thanksref{t2}\ead[label=e2]{eugenio.regazzini@unipv.it}}

\thankstext{t1}{Work partially supported by MIUR-2012AZS52J-003 and by INdAM-GNAMPA Project 2014.
}
\thankstext{t2}{Also affiliated with CNR-IMATI, Milano, Italy. Work partially supported by MIUR-2012AZS52J-003.
}
\runauthor{E. Perversi and E. Regazzini}

\affiliation{Universit\`a degli Studi di Pavia}

\address{Universit\`a degli Studi di Pavia, Dipartimento di Matematica, via Ferrata 1
27100 Pavia, Italy\\
\printead{e1}\\
\phantom{E-mail:\ }\printead*{e2}}

\end{aug}

\begin{abstract}
For a general inelastic Kac-like equation recently proposed, this paper studies the long-time behaviour of its probability-valued solution. In particular, the paper provides necessary and sufficient conditions for the initial datum in order that the corresponding solution converges to equilibrium. The proofs rest on the general CLT for independent summands applied to a suitable Skorokhod representation of the original solution evaluated at an increasing and divergent sequence of times. It turns out that, roughly speaking, the initial datum must belong to the standard domain of attraction of a stable law, while the equilibrium is presentable as a mixture of stable laws. 
\end{abstract}



\end{frontmatter}

\section{Introduction}\label{sec:intro}
There has been recent interest in some inelastic counterparts of the Kac one-dimensional Boltzmann-like equation. Particular attention has been paid to the convergence to equilibrium as, for example, in \cite{BaLaSelf,BaLaLD,BaLaMa,BaLaRe,BaPe,BonPerReg,GabReg2012,PulTos}. (See also \cite{BaLaMa3D} and references therein for multidimensional inelastic models.) Typically, one considers suitable initial data (in the form of probability laws) and proves that the ensuing solutions converge weakly to some distinguished probability distributions (p.d.'s, for short). Furthermore, remarkable efforts have been made to discover the rate of approach to equilibrium, exactly as in allied works on the classical kinetic equations such as \cite{CarlenCarvalhoGabetta2005,CarGabReg2008,DoleraGabettaRegazzini,DoleraRegazzini2010,GabReg2006,GabReg2008,GabettaRegazzini2010,McKean,McKean67} for one-dimensional models, and \cite{CarGabReg2007,CarGabTos,DoleraRegazzini2012} for multidimensional ones. New problems arise in connection with the following natural questions: Can one formulate necessary conditions on the initial data in order that they produce relaxation to equilibrium? Can one determine necessary and sufficient conditions for the same purpose? In the present paper, the former question is solved for a general form of the model introduced in \cite{BaLaMa}, while the latter is addressed with respect to a few specific cases including, in any case, the one that has been considered repeatedly. 

In the aforesaid general model, if $(v,w)$ and $(v',w')$ indicate pre-collisional and post-collisional velocities, respectively, of two colliding particles, one assumes that
\begin{equation}\label{collisioni}
\left\{ \begin{aligned}
&v'= \tL_1 v+ \tR_1 w\\
&w'= \tR_2 v+ \tL_2 w
\end{aligned}
\right.
\end{equation}
where $(\tL_1,\tR_1)$ and $(\tL_2,\tR_2)$ are random vectors in $\R^2$ with common p.d. $\tau$. In the rest of the paper, it is supposed that the function $\CS$ defined by
\begin{equation}\label{S}
\CS(p)=\int_{\R^2}(|l|^p+|r|^p)\tau(dl dr)-1\qquad(p\geq0)
\end{equation}
satisfies the condition:
\begin{equation}\label{ip1}
\begin{split}
&\text{The equation }\CS(p)=0\text{ admits at least one solution on $(0,+\infty)$} \\
& \text{and $\a$ will denote the smallest root.}
\end{split}
\end{equation}
Moreover, it is also assumed that 
$\tau$ has continuous marginals and
\begin{equation}\label{ip2}
(x_0,y_0)\in \text{supp}(\tau) \quad\text{ whenever  }\quad |x_0|^\a+|y_0|^\a=1.
\end{equation}
It is worth mentioning that both the Kac equation in \cite{Kac} and its inelastic direct counterpart given in \cite{PulTos} satisfy \eqref{ip1}-\eqref{ip2} for suitable $\a$ in $(0,2]$ and $\tau$. At this stage, the ensuing kinetic equation reads
\begin{equation}\label{eq}
\partial_t \mu_t +\mu_t= Q^+(\mu_t)\qquad(t\geq0)
\end{equation}
where $\mu_t$ is a time-dependent probability measure (p.m., for short) on the real line and $Q^+(\mu_t)$ is the p.m. specified by the Fourier-Stieltjes transform
\[
\widehat{Q^+(\mu_t)}(\xi):= \int_{\R^2}\vphi(t,l\xi)\vphi(t,r\xi)\tau(dl dr)
\]
where $\vphi(t,\xi):=\int_{\R}e^{i\xi v}\mu_t(dv)$ is the Fourier-Stieltjes transform $\widehat{\mu_t}$ of $\mu_t$. Following the terminology adopted for kinetic equations, $Q^+(\mu_t)$ can be seen as a sort of $2$-fold \textit{Wild convolution} of $\mu_t$ with itself. (See \cite{McKean}.) It is also important to recall that the Cauchy problem, obtained from the combination of \eqref{eq} with any initial p.d. $\mu_0$, has a unique solution. The proof of this fact is immediate by mimicking the argument used in \cite{BaLaMa}, where the support of $\tau$ is assumed to be a subset of $[0,+\infty)^2$. 

Returning to the original questions, complete answers have been given, until now, only for the Kac equation \cite{CarGabReg2008,GabReg2008} and for its direct inelastic counterpart \cite{BonPerReg,GabReg2012}. In both these cases, in order for the solution to converge, it is necessary and sufficient that the symmetrized form $\mu^*_0$ of the initial p.d. $\mu_0$ characterized through the probability distribution function (p.d.f., for short), by
\[
F^*_0(x)=\mu_0^*(-\infty,x]=\{\mu_0(-\infty,x]+\mu_0[-x,+\infty)\}/2\quad(x\in\R),
\]
belong to the \textit{standard domain of attraction} (s.d.a., for short) of the $\a$-stable distribution having Fourier-Stieltjes transform $\xi\mapsto e^{-k_\a|\xi|^\a}$, for some $k_\a$ in $[0,+\infty)$. This s.d.a. has to be meant as the class of all the p.m.'s with p.d.f.'s $F$ satisfying either
\begin{equation}\label{sda<2}
\begin{split}
&\lim_{x\to+\infty}x^\a F(-x)=c_1\text{ and }\lim_{x\to+\infty}x^\a[1-F(x)]=c_2\\
&\text{for some non-negative $c_1$ and $c_2$,}\text{ if }0<\a<2,
\end{split}
\end{equation}
or
\begin{equation}\label{sda2}
\begin{split}
&\int_{\R}x^2 dF(x)<+\infty\qquad\text{ if }\a=2.
\end{split}
\end{equation}

The limiting behaviour of the solution to \eqref{eq}, when \eqref{ip1} is in force for some $\a$ greater than $2$, is considered, for the first time, in Theorem \ref{Tha>2} of the present paper. Moreover, in line with the results obtained for the above-mentioned special cases, it will be proved that the symmetrized initial datum $\mu^*_0$ for \eqref{eq} must satisfy one of the two conditions \eqref{sda<2}-\eqref{sda2} $-$ with $F$ replaced by $F^*_0$ $-$ in order for the solution of the Cauchy problem to converge weakly. As far as sufficiency is concerned, it is shown that situations in which these conditions on $F^*_0$ turn out to be also sufficient, for example when $\tau$ is invariant w.r.t. $(\pi/2)$-rotations, coexist with others in which it is required that $\mu_0$ itself belongs to a specific s.d.a., for example when the support of $\tau$ is contained in $[0,+\infty)^2$. Of course, if $\mu_0$ is an element of some s.d.a., then $\mu^*_0$ is such, but $\mu^*_0$ can belong to some s.d.a. even if $\mu_0$ does not. (See the example in Appendix \ref{AppExample}.)

The rest of the paper is organized as follows. In Subsection \ref{sec:preliminaries}, a probabilistic representation of the solution to the general Cauchy problem is recalled. The new results are carefully formulated in Subsection \ref{sec:main}, while their proofs are deferred to Section \ref{sec:proof}. From a methodological viewpoint, it seems appropriate to highlight Lemma \ref{lemma2}, since it represents the key point in the proof of the main results. 

\section{Preliminaries and statement of the main results}\label{sec:prel&main}
Before providing a precise formulation for the new results, some facts about a probabilistic representation of the solution of the Cauchy problem are recalled. They are useful to grasp the connections of the issues raised in the introduction with the central limit problem.

\subsection{Probabilistic representation of the solution to the Cauchy problem}\label{sec:preliminaries}
This representation follows from associating a stochastic model with a system of many molecules colliding in pairs, in such a way that this very same model turns out to be consistent with the Wild-McKean representation of the solution of \eqref{eq}. (See $(8)$-$(9)$ in \cite{BaLaMa} for this point.) All the random elements one is about to consider are supposed to be defined as measurable functions on some measurable space $(\O,\CF)$, in such a way that their p.d.'s turn out to be images of a p.m. $\CP$ supported by $(\O,\CF)$. One starts by considering a distinguished molecule, and defines $\tnu_t$ to be the random number of particles that "contribute", according to the following scheme, to the velocity $V_t$ of the observed molecule at time $t$, for any $t>0$. (Cf. \cite{CarlenCarvalho}.) This idea of "contribution" is illustrated in the following example: Consider particles $1,2,\dots,7$ with initial velocities $X_1,\dots,X_7$, and assume that $2$ and $3$ collide before $2$ encounters $1$, and $3$ disappears; moreover, assume that $5$ and $6$ collide before $5$ encounters $4$, and $6$ disappears; to continue, suppose that $4$ collides with $7$, $5$ disappears and afterwards $4$ encounters $1$, while $7$ disappears. According to \eqref{collisioni}, due to this specific sequence of collisions, initial velocities $X_1,\dots,X_7$ change as follows: the first collision between $2$ and $3$ yields post-collisional velocities $X'_2=\tL_1 X_2+\tR_1 X_3$ and $X'_3=\tL_2 X_3+\tR_2 X_2$. Thus, the velocity of $2$ immediately before encountering $1$ is given by $X'_2$ and, hence, the post-collisional velocity
of $1$ is $X'_1=\tL_3 X_1+\tR_3 X'_2$. In the meantime, $5$ encounters $6$ and changes its velocity $X_5$ into $X'_5=\tL_4 X_5+\tR_4 X_6$ so that the velocity of $4$ immediately after the collision with $5$ is given by $X'_4=\tL_5 X_4+\tR_5 X'_5$. At this stage, this velocity changes, due to the collision with $7$, into $X''_4=\tL_6 X'_4+\tR_6 X_7$. Finally, the collision between $1$ and $4$ occurs with pre-collisional velocities $X'_1$ and $X''_4$, respectively, and then the post-collisional velocity of $1$ is $X''_1=\tL_7 X'_1+\tR_7 X''_4=\tL_7\tL_3 X_1+ \tL_7\tR_3\tL_1 X_2+\tL_7\tR_3\tR_1X_3+ \tR_7\tL_6\tL_5 X_4+\tR_7\tL_6\tR_5\tL_4 X_5+\tR_7\tL_6\tR_5\tR_4 X_6+\tR_7\tR_6 X_7$. This formula clarifies the meaning of the aforementioned term "contribution": $1$'s contribution to $V_t$ is $\tL_7\tL_3 X_1$, $2$'s contribution is $\tL_7\tR_3\tL_1 X_2$ and so on up to particle $7$, since $\tnu_t=7$. The above description can be visualized through the McKean tree of Figure 1.
\begin{figure}[h]
  \centering
  \setlength{\unitlength}{0.07\textwidth}
  \begin{picture}(7,5.5)(-3.5,-5)
    \thicklines
	\put(0,0){\circle*{0.3}}
    \put(0,0.3){$0$}    
    \put(0,0){\line(2,-1){2}} 
    \put(0,0){\line(-2,-1){2}} 
    \put(2,-1){\circle*{0.3}}
    \put(-2,-1){\circle*{0.3}}
    \put(-1.35,-0.4){$\tL_7$}
    \put(0.9,-0.4){$\tR_7$}
    \put(-2,-1){\line(1,-1){1}}
    \put(-2,-1){\line(-1,-1){1}}
    \put(-3,-2){\circle*{0.3}}
    \put(-3.3,-2.4){$1$}
    \put(-1,-2){\circle*{0.3}}    
    \put(-3.0,-1.5){$\tL_3$}
    \put(-1.45,-1.5){$\tR_3$}
    \put(2,-1){\line(1,-1){1}}
    \put(2,-1){\line(-1,-1){1}}
    \put(-1.8,-2.7){$\tL_1$}
    \put(-0.6,-2.7){$\tR_1$}
    \put(3,-2){\circle*{0.3}}
    \put(1,-2){\circle*{0.3}} 
    \put(2.55,-1.5){$\tR_6$}
    \put(1,-1.5){$\tL_6$}
    \put(3.2,-2.4){$7$}
    \put(-1,-2){\line(1,-2){0.6}}
    \put(-1,-2){\line(-1,-2){0.6}}
    \put(-0.4,-3.3){\circle*{0.3}}
    \put(-1.6,-3.3){\circle*{0.3}}
    \put(-1.9,-3.8){$2$}
    \put(-0.6,-3.8){$3$}
    \put(1,-2){\line(1,-2){0.6}}
    \put(1,-2){\line(-1,-2){0.6}}
    \put(0.2,-2.7){$\tL_5$}
    \put(1.4,-2.7){$\tR_5$}
    \put(0.2,-3.8){$4$}
    \put(0.4,-3.3){\circle*{0.3}}
    \put(1.6,-3.3){\circle*{0.3}}
    \put(1.6,-3.3){\line(1,-2){0.6}}
    \put(1.6,-3.3){\line(-1,-2){0.6}}
    \put(0.95,-4.6){\circle*{0.3}}
    \put(2.25,-4.6){\circle*{0.3}}
    \put(0.7,-4.1){$\tL_4$}
    \put(2.05,-4.1){$\tR_4$}
    \put(0.55,-5){$5$}
    \put(2.4,-5){$6$}
\end{picture}
\caption{}
\end{figure}\newline
The leaves are labelled, from left to right, by the labels $1,\dots,7$ of the seven particles contributing to $V_t$. This way, the contribution of each particle $j$ can be read as the product of the $\tL$'s and $\tR$'s one finds on the path which connects $j$ with the root node $0$. In general, there is a one-to-one correspondence between McKean trees and systems of many molecules colliding (in pairs) and changing their states according to \eqref{collisioni}. Therefore, both the number and the entities of the "contributions" to $V_t$ can be easily determined by resorting to trees like in Figure 1.

In point of fact, for the sake of realism, one thinks of these systems as governed by probabilistic, rather than deterministic, laws. In particular, $\tnu:=(\tnu_t)_{t\geq0}$ is seen as a pure birth (Yule-Furry) process on $\{1,2,\dots\}$ having birth rate $\lambda_n=n$ for each $n\geq1$, and the unit mass at $1$ as initial distribution. Then, $\CP\{\tnu_t=n\}$ coincides with the probability that the process starting at $1$ will be in state $n$ at time $t$, that is,
\[
\CP\{\tnu_t=n\}=e^{-t}(1-e^{-t})^{n-1}\qquad\qquad (n=1,2,\dots,\;t\geq0)
\]
with the proviso that $0^0=1$. As to the McKean trees, i.e., as to all possible forms of collision systems, one resorts to a discrete-parameter Markov chain $\ta=\{\ta_n:n=1,2,\dots\}$ defined as follows: The state space of $\ta$ is the set of all McKean trees and, for each value $n$ of the parameter, $\ta_n$ has to be meant as a random tree with $n$ leaves. The initial distribution, that is, the p.d. of $\ta_1$, is assumed to be the unit mass at the unique tree with one leaf. For every $a_n$ in the set $\CA_n$ of all trees with $n$ leaves $(n=1,2,\dots)$, the transition probabilities are specified as
\[
\CP\{\ta_{n+1}=a_{n+1}|\ta_n=a_n\}=\dfrac{1}{n}\I_{G(a_n)}(a_{n+1})
\]
where $G(a_n)$ denotes the set of the trees in $\CA_{n+1}$ obtained by attaching a copy of $\ta_2$ (which, of course, is non-random) to a specific leaf of $a_n$. To complete the picture, one introduces a sequence $(\tL,\tR)=((\tL_n,\tR_n))_{n\geq1}$ of independent and identically distributed (i.i.d., for short) random vectors with $\tau$ as common p.d. (see \eqref{ip1}-\eqref{ip2}) and a sequence $\tX=(\tX_n)_{n\geq1}$ of i.i.d. random numbers having the initial datum $\mu_0$ as common p.d.. It remains to specify that $\tnu, \ta, (\tL,\tR),\tX$ are supposed to be stochastically independent. Now, the contribution of each particle to $V_t$ can be fully described. One considers the random tree $\ta_{\tnu_t}$, a particle $j$ in $\{1,\dots,\tnu_t\}$ and the product $\tbe_{j,\tnu_t}$ of the $\tL$'s and $\tR$'s one encounters on the path that connects $j$ with the root node of $\ta_{\tnu_t}$ (the number of the edges of such a path will be called \textit{depth} of $j$). Then, in the general case,
\begin{equation}\label{V_t}
 V_t=\sum_{j=1}^{\tnu_t}\tbe_{j,\tnu_t}X_j\qquad(t\geq0).
\end{equation}
The connection between this stochastic model and the Cauchy problem attached to the Boltzmann-like equation \eqref{eq} lies in the important fact that the p.d. of $V_t$ satisfies \eqref{eq} with initial datum $\mu_0$. (See Proposition 1 in \cite{BaLaMa}.) Of course, recalling that $\vphi(t,\cdot)$ stands for the Fourier-Stieltjes transform of the solution $\mu_t$, and putting $\vphi_0(\cdot):=\vphi(0,\cdot)=\widehat{\mu_0}(\cdot)$ one gets 
\begin{equation*}
\begin{split}
\vphi(\xi,t)=\int_{\Omega}\prod_{j=1}^{\tnu_t(\omega)}\vphi_0(\tbe_{j,\tnu_t(\omega)}(\omega)\xi)\CP(d\omega),
\end{split}
\end{equation*}
providing a disintegration of the solution to the Cauchy problem associated with \eqref{eq} into components which are the p.d.'s of weighted sums of i.i.d. random numbers having common p.d,. $\mu_0$. It is just this disintegration which connects the relaxation to equilibrium of the kinetic model with the central limit problem of the probability theory. 

To pave the way for next developments it is worthwhile to recall an interesting fact about the $\tbe$'s, proved in \cite{BaLaMa}, Subsection 2.3 and Lemma 2: \textit{$\sum_{j=1}^{\tnu_t}|\tbe_{j,\nu_t}|^\a$ converges, with probability one, to a non-negative random number $\Minf$  satisfying $\E(\Minf)=1$}, where $\E$ stands for expectation w.r.t. $\CP$.

\subsection{New results}\label{sec:main}
The new results of the present paper are chiefly concerned with necessary conditions for weak convergence of the solution to the Cauchy problem. The first one provides an essential justification for the restriction of the admissible values of $\a$ to the interval $(0,2]$. In point of fact, this restriction has been conventionally, rather than substantially, so far accepted, and the next theorem confirms its reasonableness.

\begin{theorem}\label{Tha>2}
Suppose $\tau$ satisfies \eqref{ip1}-\eqref{ip2} with $\a>2$. If the solution $\mu_t$ to the Cauchy problem associated with \eqref{eq}, initial p.d. $\mu_0$, converges weakly to a p.d. $\mu_\infty$, as $t\to+\infty$, then both $\mu_0$ and $\mu_\infty$ must be unit masses at points $x_0$ and $x_1$, respectively.

More precisely:
\begin{itemize}
\item[$(i)$] One has $\mu_\infty=\mu_0=\delta_{x_0}$ with $x_0\neq0$ if $\tau\{(x,y)\in\R^2:\;x+y-1=0\}=1$. Conversely, if $\mu_\infty=\delta_{x_1}$ with $x_1\neq0$, then $\tau\{(x,y)\in\R^2:\;x+y-1=0\}=1$ and $x_0=x_1$. 
\item[$(ii)$] One has $\mu_\infty=\delta_0$ if and only if $($at least$)$ one of the two following conditions is verified: 
\begin{itemize} 
\item [$(ii_1)$] $x_0=0$
\item [$(ii_2)$] $\sum_{j=1}^{\tnu_t}\tbe_{j,\tnu_t}$ converges in distribution to zero, as $t$ goes to infinity.
\end{itemize}
\end{itemize}
\end{theorem}

\begin{proof}
See Subsection \ref{sec:proofTha>2}.
\end{proof}

The lesson of Theorem \ref{Tha>2} is that, if $\a>2$, unit masses are the sole admissible stationary distributions. From a logical viewpoint, it is worth noticing that condition \eqref{ip1} may coexist with $(i)$. Consider, for example the case in which $\tR_1=1-\tL_1$ almost surely and the p.d. of $\tL_1$ is the Gaussian centred at $1/2$ with variance $\sigma^2$. It is easy to show that there is a value of $\sigma^2$, say $\bar{\sigma}^2$, so that $\int_{\R}|x|^\a 1/\sqrt{2\pi\bar{\sigma}^2}\exp\{-(x-1/2)^2/2\bar{\sigma}^2\}dx =1/2$. In the same vein, \eqref{ip1} and $(ii_2)$ may coexist, for example, if $\tL_1=\tR_1$ almost surely and the p.d. of $\tL_1$ is the Gaussian centred at zero with variance $1/4$.\newline

The second, and more significant, group of results is concerned with values of $\a$ in $(0,2]$, when convergence to the steady state occurs in the presence of interesting forms of initial data. 

\begin{theorem}\label{Th1}
Suppose $\tau$ satisfies both the moment and support conditions \eqref{ip1}-\eqref{ip2} for some $\a$ in $(0,2]$. If the solution $\mu_t$ to the Cauchy problem associated with \eqref{eq}, initial datum $\mu_0$, converges weakly to a p.d., as $t\to+\infty$, then 
\begin{equation}\label{CNsimm}
\lim_{x\to+\infty}x^\a[1-F^*_0(x)] \text{ exists and is finite}.
\end{equation}
\end{theorem}

\begin{proof}
See Subsection \ref{sec:proofTh1}.
\end{proof}

Like the central limit theorem for weighted sums of i.i.d. summands, one could expect that the above condition is also sufficient for the convergence of the solution. On the one hand this is the case when, for example, $\tau$ agrees either with the Kac model or with its direct inelastic counterpart (see \cite{BonPerReg,CarGabReg2008,GabReg2008,GabReg2012} and the next more general Theorem \ref{Th4}). On the other hand the conjecture fails, for example, when the support of $\tau$ is a subset of $[0,+\infty)^2$, which is the version of \eqref{eq} most widely studied so far. To deal with this case, some additional notation is needed together with the following reformulation of \eqref{sda<2}, where $F_0$ is taken in the place of $F$ and $\a$ belongs to $(0,2]$: 
\begin{equation}\label{NDA}
\lim_{x\to+\infty}x^\a F_0(-x)=c_1\text{ and }\lim_{x\to+\infty}x^\a[1-F_0(x)]=c_2,
\end{equation}
where $c_1$ and $c_2$ are non-negative numbers. Moreover, let $m_{0,i}:=\int_{\R}x^i\mu_0(dx)$ for $i=1,2$ and let $\widehat{\psi_\a}(\cdot\;;\chi,k_\a,\gamma)$ indicate the Fourier-Stieltjes transform of the stable law $\psi_\a(\cdot\;;\chi,k_\a,\gamma)$, i.e.,
\[
\widehat{\psi_\a}(\xi;\chi,k_\a,\gamma)=\exp\Big\{i\chi\xi-k_\a|\xi|^\a\Big(1-i\gamma \dfrac{\xi}{|\xi|}\o(\xi,\a)\Big)\Big\},
\]
where 
\[
\begin{aligned}
\o(\xi,\a)&= \tan (\pi\a/2) & \a\neq1\\
& =2\pi^{-1}\log|\xi| & \a=1
\end{aligned}
\]
with the proviso that: \newline
\begin{itemize}
\item
$\chi:=0$, $k_\a:=[2\Gamma(\alpha)\sin(\pi\alpha/2)]^{-1}(c_1+c_2)\pi$ and $\gamma:=\I_{\{c_1+c_2>0\}}(c_2-c_1)/(c_1+c_2)$, if $\a\in(0,1)\cup(1,2)$;
\item $\chi:=\eta-\int_{\R\setminus\{0\}}[-\I_{(-\infty,-1]}(y)+y\I_{(-1,1]}(y)+\I_{(1,+\infty)}(y)-\sin y]\nu(dy)$, $k_1:=(c_1+c_2)\pi/2$ and $\gamma:=0$, when $\a=1$, $\eta$ and $\nu$ being characteristic parameters of the L\'evy-Khinchin representation of $\widehat{\psi_\a}$ according to Proposition 11 in Chapter 17 and Theorem 30 in Chapter 16 of \cite{FristedtGray}, with $Q_{1,n}=Q_{2,n}=\dots=\mu_0$; 
\item $\chi:=0$, $k_2:=(m_{0,2}-m_{0,1}^2)/2$ and $\gamma:=0$, if $\a=2$.
\end{itemize}

Coming back to the discussion about sufficiency, one starts by giving a more complete version of Theorems 1-3 in \cite{BaLaMa} and Theorem 2.3 in \cite{BaPe} with respect to the case of $\a=1$.

\begin{theorem}\label{ThSuff}
Suppose $\tau$ satisfies \eqref{ip1} for some $\a$ in $(0,2]$, with $supp(\tau)$ $\subset [0,+\infty)^2$, and 
\begin{equation}\label{p}
\CS(p)<0\qquad\text{for some }p>0.
\end{equation}
Then, the solution $\mu_t$ to the Cauchy problem associated with \eqref{eq}, initial datum $\mu_0$, converges weakly to a p.m. $\mu_\infty$, as $t\to+\infty$, if: 
\[
\begin{aligned}
&\text{condition \eqref{NDA} holds whenever $\a\in(0,1)$;}\\
&\text{condition \eqref{NDA} along with $c_1=c_2$ are met whenever $\a=1$;}\\
&\text{condition \eqref{NDA} and $m_{0,1}=0$ are in force whenever $\a\in(1,2)$;} \\
&\text{$m_{0,1}=0$ and $m_{0,2}<+\infty$ are valid whenever $\a=2$.} 
\end{aligned}
\]
Furthermore, the Fourier-Stieltjes transform of the limiting p.d. $\mu_\infty$ is given by
\begin{align}
&\int_{0}^{+\infty}\widehat{\psi_\a} (\xi m^{1/\a}; 0,k_\a ,\gamma)\nu_\a(dm) & \text{whenever }&\a\in(0,1)\cup(1,2)\notag\\
&\int_{0}^{+\infty}\widehat{\psi_1} (\xi m; \chi , k_1 ,0)\nu_1(dm) &\text{whenever }&\a=1\notag\\
&\int_{0}^{+\infty}\widehat{\psi_2} (\xi m^{1/2};0,k_2 ,0)\nu_2(dm)&\text{whenever }&\a=2\notag
\end{align}
for every $\xi\in\R$, where $\nu_\a$ is the p.d. of $\Minf$ for each $\a$ in $(0,2]$, and $\chi$, $k_\a$, $\gamma$ are the same as in the itemization preceding the theorem.
\end{theorem}

\begin{proof}
See Subsection \ref{sec:proofThSuff}.
\end{proof}

Before presenting necessary and sufficient conditions, the following miscellaneous remarks could be in order.
\begin{remark}\label{remark1}
{\rm In view of Proposition 2 in \cite{BaLaMa}, one recalls that, when \eqref{p} is in force and $\a$ is the unique root of equation \eqref{ip1}, the p.d. $\nu_\a$ admits moments of any order. Moreover, if a second root $\theta$ exists, the only finite moments of $\nu_\a$ are those of order strictly smaller than $\theta/\a$. Combination of these facts with the well-known moment properties of the stable laws yields: If $\a<2$, then $\mu_\infty$ admits the $p$-th moment if and only if $p<\a$. If $\a=2$, then the $p$-th moment of $\mu_\infty$ is finite for $p>2$ if and only if $p<\theta$. }
\end{remark}
\begin{remark}\label{remark2}
{\rm An interesting problem is the search of conditions under which the limiting law $\mu_\infty$ is a (pure) stable law. A complete answer can be obtained by combining Theorems 1, 3 in \cite{BaLaMa} with Theorem 2.3 in \cite{BaPe} and allied results in \cite{AlsMei} on the fixed points of operators like $Q^+$. All things considered, one finds that $\nu_\a$ must be a point mass.}
\end{remark}
\begin{remark}\label{remark3}
{\rm Proposition 3.9 in \cite{BaPe} shows that whenever $\a$ belongs to $(0,2)$ the tails of non-degenerate limiting p.d.'s behave like $(x_0/|x|)^{\a}$ as $x\to\infty$, for suitable $x_0$. According to \cite{Mandelbrot}, these p.d.f.'s can be called \textit{weak Pareto laws}. Furthermore, limiting point masses arise when $c_1=c_2=0$.  }
\end{remark}
\begin{remark}\label{remark4}
{\rm As for the case of $\a=2$, in view of {\rm Remark \ref{remark1}},  it is worth distinguishing the following two subcases: If $\a$ is the sole root of equation \eqref{ip1}, since in this case $\mu_\infty$ has moments of any order, one can conclude that $F_\infty(-x)=1-F_\infty(x)=o(1/x^p)$ for every $p>0$, as $x\to+\infty$. On the other hand, if $\theta>\a$ is the second root of the equation in \eqref{ip1}, then $\int_{\R}|x|^p d\mu_\infty<+\infty$ for $p<\theta$ and $\int_{\R}|x|^\theta d\mu_\infty=+\infty$ (see Remark \ref{remark1}) so that, from the Markov inequality,
\[
F_\infty(-x)=1-F_\infty(x)\leq \dfrac{A_p}{x^p}\qquad(x>0)
\]
for every $p<\theta$ and $A_p:=\int_{\R}|x|^p\mu_\infty(dx)/2$. To obtain a lower bound for $F_\infty(-x)$ one notes that putting $G(y):=[2 F_\infty(y^{1/\theta})-1]\I_{[0,+\infty)}(y)$, one has $\int_{0}^{+\infty}y dG(y)=+\infty$. From Proposition 3.3. in \cite{CifReg} one obtains $\lim_{y\to+\infty}(1-G(y))g(y)=+\infty$ for every function $g$ continuous, strictly increasing and positive on $(a,+\infty)$ for some $a>0$ such that $\int_a^{+\infty}\{1/g(y)\}dy<+\infty$. Thus, choosing, for example, $g(y):=y(\log y)^{1+\delta}$ for $y>1$ and $\delta>0$, for every positive $M$ there exists $\bar{y}$ for which $1-G(y)\geq M/[y(\log y)^{1+\delta}]$ holds for every $y\geq\bar{y}$. Finally, from the definition of $G$, one gets
\[
F_\infty(-x)=1-F_\infty(x)=\dfrac{1}{2}(1-G(x^\theta))\geq \dfrac{M}{2 \theta^{1+\delta}}\dfrac{1}{x^\theta(\log x)^{1+\delta}}
\]
for every $x\geq \bar{y}^{1/\theta}$.}
\end{remark}

Resuming the main line of discussion, the way is paved for presenting necessary and sufficient conditions for the relaxation to equilibrium.

\begin{theorem}\label{Th2}
Suppose $supp(\tau)\subset [0,+\infty)^2$ and, for some $\a$ in $(0,2]$, \eqref{ip1}-\eqref{ip2} and \eqref{p} are in force. Then, the solution $\mu_t$ to the Cauchy problem associated with \eqref{eq}, initial p.d. $\mu_0$, converges weakly to a p.m. $\mu_\infty$, as $t\to+\infty$, if and only if
\[
\begin{aligned}
&\text{condition \eqref{NDA} holds whenever $\a\in(0,1)$;}\\
&\text{condition \eqref{NDA} along with $c_1=c_2$ are met whenever $\a=1$;}\\
&\text{condition \eqref{NDA} and $m_{0,1}=0$ are in force whenever $\a\in(1,2)$;} \\
&\text{$m_{0,1}=0$ and $m_{0,2}<+\infty$ are valid whenever $\a=2$.} 
\end{aligned}
\]
Of course, $\mu_\infty$ is the same as in {\rm Theorem \ref{ThSuff}}.
\end{theorem}

\begin{proof}
See Subsection \ref{sec:proofTh2}.
\end{proof}

The next theorem, which includes both the Kac model and its inelastic counterpart, provides an example in which the weaker condition involved in Theorem \ref{Th1} turns out to be sufficient for the convergence. 

\begin{theorem}\label{Th4}
Suppose $\tau$ is invariant w.r.t. $(\pi/2)$-rotations and, for some $\a$ in $(0,2]$, \eqref{ip1}-\eqref{ip2} and \eqref{p} are in force. Then, \eqref{CNsimm} $[m_{0,2}<+\infty$, respectively$]$ is necessary and sufficient whenever $\a$ belongs to $(0,2)$ $[\a=2$, respectively$]$ in order that the solution $\mu_t$ to the Cauchy problem associated with \eqref{eq}, initial p.d. $\mu_0$, converge weakly to a p.m. $\mu_\infty$, as $t\to+\infty$.
\end{theorem}

\begin{proof}
See Subsection \ref{sec:proofTh4}.
\end{proof}

One notes the redundancy of assumption \eqref{ip2} in the sufficiency part of the last theorem.

\section{Proofs}\label{sec:proof}
Some crucial parts of the proofs are based on the \textit{Skorokhod representation} for sequences which converge in law. (See, e.g., Theorem 6.7 in \cite{Billingsley}.) It is worth recalling such a representation for the sake of expository clarity, even if analogous descriptions have already been given in \cite{BaLaRe,BonPerReg,DoleraRegazzini2012,ForLadReg,GabReg2012}. Before proceeding to apply the Skorokhod theorem, it is useful to introduce some slight changes to the presentation in Subsection \ref{sec:preliminaries}. In particular, one replaces the probability space $(\O,\CF,\CP)$ with the family $\{(\O,\CF,\CP_t):\;t\geq0\}$. The random elements $(\tL,\tR)$, $\tX$ are maintained, while the roles of $\tnu$ and $\ta$ are respectively played by:
\begin{itemize}
\item A random number $\tnu$ taking values in $\N$ whose p.d., under $\CP_t$, is given by $\CP_t\{\tnu=n\}=e^{-t}(1-e^{-t})^{n-1}$ for every $n\geq1$ and $t\geq0$.
\item A sequence $\ti:=(\ti_n)_{n\geq1}$ of integer-valued random numbers which, under $\CP_t$, are independent, each $\ti_n$ being uniformly distributed on $\{1,\dots,n\}$, for every $t\geq0$.
\end{itemize}

It is easy to verify that each realization of the sequence $\ti$ specifies a McKean tree, and that the distributional properties of $\ti$ agree with the Markov structure of the law of $\ta$. 

According to Subsection \ref{sec:preliminaries}, the random elements $\tnu, \ti, (\tL,\tR), \tX$ are assumed to be \textit{stochastically independent} under each $\CP_t$. 

These points accepted, one introduces a random vector $W$, which contains all the elements that characterize convergence in agreement with the general form of the central limit theorem, 
\[
W=W(\o):=(\tnu(\o),\ti(\o),(\tL(\o),\tR(\o)), \tbe(\o), \tlm(\o),\tLm(\o),\tM(\o),\tu(\o))
\]
for every $\o$ in $\O$, with:
\begin{itemize}
\item The same $\tbe=(\tbe_{j,n}:\;j=1,\dots,n)_{n\geq1}$ as in Subsection \ref{sec:preliminaries}, that can now be expressed through the following recursive relation
\begin{equation}\label{betaricorsivo}
\begin{array}{ll}
&\tbe_{1,1}=1\\
&(\tbe_{1,n+1},\dots,\tbe_{n+1,n+1})=(\tbe_{1,n},\dots,\tbe_{\ti_n-1,n},\tbe_{\ti_n,n}\tL_n,\tbe_{\ti_n,n}\tR_n, \\
&\qquad\qquad\qquad\qquad\qquad\qquad\qquad\tbe_{\ti_n+1,n},\dots,\tbe_{n,n})\quad(n\geq1).
\end{array}
\end{equation}
\item $\tlm=\tlm(\o):=(\tlm_1(\o),\dots,\tlm_{\tnu(\o)}(\o),\delta_0,\delta_0,\dots)$ where, for each $j$ in $\{1,\dots,$ $\tnu(\o)\}$, $\tlm_j(\o)$ is the p.d. determined by the characteristic function $\xi\mapsto\vphi_0(\tbe_{j,\tnu(\o)}(\o)\xi)$, $\xi\in\R$.
\item $\tLm=$convolution of the elements of $\tlm$.
\item $\tM(\o)$ is the p.d. of $\sum_{j=1}^{\tnu(\o)}|\tbe_{j,\tnu(\o)}|^\a$, where $\tbe_{j,\tnu(\o)}$ is the same as $\tbe_{j,n}$ with $n=\tnu(\o)$.
\item $\tu:=(\tu_k)_{k\geq1}$, with $\tu_k=\max_{1\leq j\leq\tnu}\tlm_j([-\frac{1}{k},\frac{1}{k}]^c)$ for every $k\geq1$.
\end{itemize}
Introducing the symbol $\CP(M)$ to denote the set of all p.m.'s on the Borel class $\CB(M)$ of a metric space $M$, one can say that the range of $W$ is a subset of 
\[
S:=\overline{\N}\times \overline{\N}^\infty\times (\overline{\R}^2)^\infty\times \overline{\R}^\infty\times (\CP(\overline{\R}))^\infty\times \CP(\overline{\R})\times \CP(\overline{\R})\times[0,1]^\infty.
\]
Here, $\CP(\overline{\R})$ is metrized consistently with the topology of weak convergence of p.m.'s so that it can be seen as a separable, compact and complete metric space. Thus, $S$ can be metrized so that it results in a separable, compact and complete metric space (Theorems 6.2, 6.4 and 6.5 in Chapter 2 of \cite{Parthasarathy}). Obviously, the family of p.m.'s $\{\CP_tW^{-1}:t\geq0\}$ is \textit{uniformly tight} on $\CB(S)$, and any subsequence from this family contains a weakly convergent subsequence $Q_n:=\CP_{t_{n}}W^{-1}$ with $0\leq t_1<t_2<\dots$ and $t_n \nearrow +\infty$. Hence, the Skorokhod's representation theorem can be applied to state the existence of a probability space $\Big(\cO,\cF,\cP\Big)$ and of a sequence of $S$-valued random elements 
\[
\cW_n:=(\cnun,\cin,(\cLn,\cRn),\cben,\clmn,\cLmn,\cMn,\cun),\qquad n\geq1
\]
defined on $\cO$ so that:
\begin{itemize}
\item The p.d. of $\cW_n$ is $Q_n$, for every $n$.
\item $\cW_n$ converges pointwise to a random element $\cW$ whose p.d. is the weak limit of $(Q_n)_{n\geq1}$.
\end{itemize}
From the first point, the equalities
\begin{equation}\label{ricorrenzabeta}
\begin{array}{ll}
&\cben_{1,1}=1\\
& (\cben_{1,k+1},\dots,\cben_{k+1,k+1})=(\cben_{1,k},\dots,\cben_{\cin_k-1,k},\cLn_k\cbe_{\cin_k,k},\\
&\qquad\qquad\qquad\qquad\qquad\qquad\qquad\cRn_k\cben_{\cin_k,k},\cben_{\cin_k+1,k},\dots,\cben_{k,k})
\end{array}
\end{equation}
are met for every $k$ and $n$, with $\cP$-probability $1$. This paves the way for two preparatory lemmas in which $\CM([1,+\infty))$ represents the set of all the finite measures on $\CB([1,+\infty))$.

\begin{lemma}\label{lemma1}
If $\nu\colon \cO\rightarrow\CM([1,+\infty))$ is a random finite measure, then there exists a countable subset $\CI$ of $[1,+\infty)$ such that, for every $x_0$ in $\CI^c\cap (1,+\infty)$, $\nu\{x_0\}(\co)=0$ for every $\co$ in a subset $\cO_{x_0}$ of $\cO$ with $\cP(\cO_{x_0})=1$.
\end{lemma}

\begin{proof}
See Section \ref{prooflemma1} in Appendix \ref{sec:Appendix1}.
\end{proof}

\begin{lemma}\label{lemma2}
If $\a$ agrees with \eqref{ip1}-\eqref{ip2}, then for every $\veps>0$ and every strictly increasing and divergent sequence $(y_n)_{n\geq1}$ such that 
\begin{equation}\label{y_n}
y_1^\a>\dfrac{1}{\veps}\qquad \text{and}\qquad y_{n+1}^\a>\dfrac{1}{\veps}\sum_{j=1}^n (y_j^\a+1) \qquad(n\geq1),
\end{equation}
there exist:
\begin{itemize}
\item a point $\co_0$ in $\cO$,
\item an integer-valued, strictly increasing and divergent sequence $(N_n)_{n\geq0}$, with $N_0:=1$,
\item a sequence of sets $(\CR_n)_{n\geq1}$ with $\CR_n\subset\{1,\dots,N_n\}$ and $|\CR_n|=N_{n-1}$ for every $n\geq1$,
\item an array $(\delta_k^{(n)})_{n\geq1, k=1,\dots,n}$ of positive real numbers
\end{itemize}
for which $\cnun(\co_0)=N_n$ and
\begin{equation}\label{betaLemma}
\begin{array}{ll}
|\cben_{j,\cnu^{(k)}(\co_0)}(\co_0)|\in \Big[\dfrac{1}{y_k+\delta_k^{(n)}},\dfrac{1}{y_k}\Big] & \text{if $k$ is odd}\\
|\cben_{j,\cnu^{(k)}(\co_0)}(\co_0)|\in \Big[\dfrac{1}{y_k},\dfrac{1}{y_k-\delta_k^{(n)}}\Big] & \text{if $k$ is even}
\end{array}
\end{equation}
for every $n\geq1$, $k=1,\dots,n$ and for every $j\notin \CR_k$. Moreover, for every $n\geq1$ and every odd number $k$ in $\{1,\dots,n\}$, 
\begin{equation}\label{sommaLemma}
\sum_{j\in\CR_k}|\cben_{j,\cnu^{(k)}(\co_0)}(\co_0)|^\a<\veps.
\end{equation}
\end{lemma}

\begin{proof}
See Section \ref{prooflemma2} in Appendix \ref{sec:Appendix1}.
\end{proof}

\subsection{Proof of {\rm Theorem \ref{Th1}}}\label{sec:proofTh1}
With reference to the Skorokhod representation, assuming that $\mu_t$ converges weakly as $t\to+\infty$ is equivalent to saying that the p.d. $\cLmn(\co)$ converges weakly to a p.d., as $n\to+\infty$, for every $\co$ in $\cO$. Then, by the central limit theorem (see, for example, (16.36) in \cite{FristedtGray}), there exists a random L\'evy measure $\nu=\nu(\co)$ such that 
\[
\nu(\co)[x,+\infty)=\lim_{n\to+\infty}\sum_{j=1}^{\cnun(\co)}\clmn_j(\co)[x,+\infty)
\]
and
\[
\nu(\co)(-\infty,-x]=\lim_{n\to+\infty}\sum_{j=1}^{\cnun(\co)}\clmn_j(\co)(-\infty,-x]
\]
hold for every $\co$ in $\cO$ and for every $x>0$ with $\nu(\co)\{x\}=\nu(\co)\{-x\}=0$. Now, in view of the definitions given at the beginning of this section, 
\[
\begin{split}
&\clmn_j(\co)[x,+\infty)\\
&\qquad\qquad=\Big[1-F_0\Big(\dfrac{x}{|\cben_{j,\cnun(\co)}(\co)|}\Big)+\mu_0\Big\{\dfrac{x}{|\cben_{j,\cnun(\co)}(\co)|}\Big\} \Big]\I_{\{\cben_{j,\cnun(\co)}(\co)>0\}}\\
&\qquad\qquad+F_0\Big(-\dfrac{x}{|\cben_{j,\cnun(\co)}(\co)|}\Big)\I_{\{\cben_{j,\cnun(\co)}(\co)<0\}}
\end{split}
\]
and, analogously,
\[
\begin{split}
&\clmn_j(\co)(-\infty,-x]\\
&\qquad\qquad=F_0\Big(-\dfrac{x}{|\cben_{j,\cnun(\co)}(\co)|}\Big)\I_{\{\cben_{j,\cnun(\co)}(\co)>0\}}\\
&\qquad\qquad+\Big[1-F_0\Big(\dfrac{x}{|\cben_{j,\cnun(\co)}(\co)|}\Big)+\mu_0\Big\{\dfrac{x}{|\cben_{j,\cnun(\co)}(\co)|}\Big\}\Big]\I_{\{\cben_{j,\cnun(\co)}(\co)<0\}}.
\end{split}
\]
Hence, recalling the definition of $F^*_0$,
\[
\begin{split}
&\clmn_j(\co)[x,+\infty)+\clmn_j(\co)(-\infty,-x]\\
&\qquad\qquad\qquad=2\Big[1-F^*_0\Big(\dfrac{x}{|\cben_{j,\cnun(\co)}(\co)|}\Big)\Big]+\mu_0\Big\{\dfrac{x}{|\cben_{j,\cnun(\co)}(\co)|}\Big\}.
\end{split}
\]
Thus,
\[
\begin{split}
&\nu(\co)(-\infty,-x]+\nu(\co)[x,+\infty)\\
&\quad=\lim_{n\to+\infty}\Big(2\sum_{j=1}^{\cnun(\co)}\Big[1-F^*_0\Big(\dfrac{x}{|\cben_{j,\cnun(\co)}(\co)|}\Big)\Big]+\sum_{j=1}^{\cnun(\co)}\mu_0\Big\{\dfrac{x}{|\cben_{j,\cnun(\co)}(\co)|}\Big\}\Big)
\end{split}
\]
is valid for every $\co$ in $\cO$ and for every $x>0$ such that $\nu(\co)\{x\}=\nu(\co)\{-x\}=0$. At this stage, for every $x>1$, one defines the random measure $\bar{\nu}\colon \cO\rightarrow \CM([1,+\infty))$ by
\[
\bar{\nu}(\co)[x,+\infty):=\nu(\co)(-\infty,-x]+\nu(\co)[x,+\infty).
\]
This way, one can apply Lemma \ref{lemma1} to state the existence of a countable subset $\CI$ of $[1,+\infty)$ such that, for every $x_0$ in $\CI^c\cap (1,+\infty)$, there exists a subset $\cO_{x_0}$ of $\cO$ with $\cP(\cO_{x_0})=1$, such that $\bar{\nu}(\co)\{x_0\}=0$ for every $\co$ in $\cO_{x_0}$. Now, keeping $x_0$ fixed, notice that, without real loss of generality (because of \eqref{ricorrenzabeta} combined with the independence of $(\cLn_1,\cRn_1),(\cLn_2,\cRn_2),\dots$ and the continuity assumption on the marginals of $\tau$) one can suppose $\cO_{x_0}$ is contained in $\{\co\in\cO:\;x \cdot|\cben_{j,\cnun(\co)}(\co)|^{-1}\notin D_{F_0},\;\forall j=1,\dots,\cnun(\co),\;n\geq1\}$ where $D_f$ denotes the discontinuity set of the function $f$. Hence,
\begin{equation}\label{limiteNu}
\bar{\nu}(\co)[x_0,+\infty)=\lim_{n\to+\infty}2\sum_{j=1}^{\cnun(\co)}\Big[1-F^*_0\Big(\dfrac{x_0}{|\cben_{j,\cnun(\co)}(\co)|}\Big)\Big] 
\end{equation}
for every $\co\in\cO_{x_0}$. Going on, one defines $I:= \liminf_{x\to+\infty}x^\a(1-F^*_0(x))$ and $S:=\limsup_{x\to+\infty}x^\a(1-F^*_0(x))$. It has to be proved that $I=S<+\infty$. Let $(i_m)_{m\geq1}$ and $(s_m)_{m\geq1}$ be increasing and divergent sequences of positive real numbers such that
\[
\lim_{m\to+\infty}i_m^\a(1-F^*_0(i_m))=I\quad\text{and}\lim_{m\to+\infty}s_m^\a(1-F^*_0(s_m))=S.
\]
Given any $\veps>0$, there are subsequences $(i'_m)_{m\geq1}\subset(i_m)_{m\geq1}$ and $(s'_m)_{m\geq1}\subset(s_m)_{m\geq1}$ so that, by defining
\[
\begin{array}{ll}
x_n&:= i'_m  \qquad\text{ if }n=2m-1\\
&:=s'_m  \qquad\text{ if }n=2m,
\end{array}
\]
one gets $x_1^\a>1/\veps$ and $x_{n+1}^\a>\sum_{i=1}^n x_i^\a /\veps$ for every $n\geq1$. Then, one puts $z_n:=x_n/x_0$ for every $n\geq1$, $x_0$ being the same as in \eqref{limiteNu}, for the purpose of bounding
\[
\sum_{j=1}^{\cnun(\co)}\Big[1-F^*_0\Big(\dfrac{x_0}{|\cben_{j,\cnun(\co)}(\co)|}\Big)\Big]
\]
for every $n$ at the point $\co=\co_0$ determined through the application of Lemma \ref{lemma2} to $(\veps$, $(z_n)_{n\geq1})$. Along the subsequence $n=2m-1$, one gets 
\[
\begin{split}
&\lim_{m\to+\infty}\sum_{j=1}^{\cnu^{(2m-1)}(\co_0)}\Big[1-F^*_0\Big(\dfrac{x_0}{|\cbe^{(2m-1)}_{j,\cnu^{(2m-1)}(\co)}(\co)|}\Big)\Big]\\
&\quad =\lim_{m\to+\infty}\Big(\sum_{j\notin\CR_{2m-1}}+\sum_{j\in\CR_{2m-1}}\Big)\Big[1-F^*_0\Big(\dfrac{x_0}{|\cbe^{(2m-1)}_{j,N_{2m-1}}(\co)|}\Big)\Big]
\end{split}
\]
\[
\begin{split}
&\quad\leq \limsup_{m\to+\infty}\Big(\sum_{j\notin\CR_{2m-1}}[1-F^*_0(x_0 z_{2m-1})]+\sum_{j\in\CR_{2m-1}}\Big[1-F^*_0\Big(\dfrac{x_0}{|\cbe^{(2m-1)}_{j,N_{2m-1}}(\co)|}\Big)\Big]\\
&\quad\qquad\times\dfrac{x_0^\a}{|\cbe^{(2m-1)}_{j,N_{2m-1}}(\co)|^\a}\cdot\dfrac{|\cbe^{(2m-1)}_{j,N_{2m-1}}(\co)|^\a}{x_0^\a}\Big)\text{\qquad\qquad\qquad\quad(in view of \eqref{betaLemma})}\\
&\quad\leq \limsup_{m\to+\infty}\Big([1-F^*_0(x_{2m-1})]\cdot |\{1,\dots,N_{2m-1}\}\setminus \CR_{2m-1}|\Big)\\
&\quad\qquad+\dfrac{S+\veps}{x_0^\a}\limsup_{m\to+\infty}\sum_{j\in\CR_{2m-1}}|\cbe^{(2m-1)}_{j,N_{2m-1}}(\co)|^\a\\
&\quad \leq \limsup_{m\to+\infty}x_{2m-1}^\a [1-F^*_0(x_{2m-1})]\dfrac{N_{2m-1}-N_{2m-2}}{x_{2m-1}^\a}\\
&\qquad\qquad\qquad\qquad\qquad\qquad\qquad\qquad\qquad\quad+\veps\dfrac{S+\veps}{x_0^\a}\quad\qquad\text{(in view of \eqref{sommaLemma})}\\
&\quad=\limsup_{m\to+\infty}i_m^\a [1-F^*_0(i_m)]\dfrac{N_{2m-1}-N_{2m-2}}{x_0^\a z_{2m-1}^\a}+\veps\dfrac{S+\veps}{x_0^\a}\\
&\quad\leq \dfrac{I+\veps (S+\veps)}{x_0^\a}
\end{split}
\]
where the last inequality holds since $(N_n-N_{n-1})/z_n^\a\leq1$ for every $n\geq1$, as shown in the proof of Lemma \ref{lemma2}. Furthermore,
\[
\lim_{m\to+\infty}\sum_{j=1}^{\cnu^{(2m)}(\co_0)}\Big[1-F^*_0\Big(\dfrac{x_0}{|\cbe^{(2m)}_{j,\cnu^{(2m)}(\co)}(\co)|}\Big)\Big]
\]
\[
\begin{split}
&\qquad \geq\limsup_{m\to+\infty}\sum_{j\notin\CR_{2m}}\Big[1-F^*_0\Big(\dfrac{x_0}{|\cbe^{(2m)}_{j,N_{2m}}(\co)|}\Big)\Big]\\
&\qquad \geq\limsup_{m\to+\infty}\sum_{j\notin\CR_{2m}}[1-F^*_0(x_0 z_{2m})]\qquad\qquad\text{(in view of \eqref{betaLemma})}\\
&\qquad = \dfrac{(N_{2m}-N_{2m-1})}{x_{2m}^\a}x_{2m}^\a[1-F^*_0(x_{2m})]\\
&\qquad= \dfrac{(N_{2m}-N_{2m-1})}{x_0^\a z_{2m}^\a}s_m^\a[1-F^*_0(s_m)]\\
&\qquad\geq \dfrac{(1-\veps)S}{x_0^\a}
\end{split}
\]
where the last inequality holds since, as shown in the proof of Lemma \ref{lemma2}, $(N_n-N_{n-1})/z_n^\a>1-\veps$ for every $n\geq1$. Now, as the limit in \eqref{limiteNu} exists and is finite, one gets
\[
\dfrac{(1-\veps)S}{x_0^\a}\leq \dfrac{I+\veps (S+\veps)}{x_0^\a}<+\infty,
\]
implying that both $I$ and $S$ are finite and $(1-\veps)S\leq I+\veps (S+\veps
)$ for every $\veps>0$, that is $I=S$.

\subsection{Proof of {\rm Theorem \ref{ThSuff}}}\label{sec:proofThSuff}
It suffices to prove the theorem for $\a=1$, since all the other cases are covered by Theorems 1 and 3 in \cite{BaLaMa}. Assuming that \eqref{NDA} is in force with $\a=1$ and $c_1=c_2$, Theorem 2.6.5 in \cite{IbrLin} can be invoked to write
\[
\vphi_0(\xi)=\exp\{i\chi\xi-k_1|\xi|(1+\psi(\xi))\} \qquad(\xi\in\R)
\]
where $\psi(\xi)=o(1)$ as $|\xi|\to 0$. It is enough to show that $\vphi_n(\xi):=\E[\exp\{i\xi$ $\times\sum_{j=1}^n\tbe_{j,n}\}]$ converges pointwise, as $n\to+\infty$, to the desired characteristic function. One starts by noting that, from Lemma 3 in \cite{BaLaMa}, given any subsequence $(n')$ of $(n)$, there exists a subsequence $(n'')$ of $(n')$ such that $\max_{j=1,\dots,n''}\tbe_{j,n''}\to 0$ almost surely. Moreover,
\[
\begin{split}
\vphi_{n''}(\xi)&=\E\Big(\prod_{j=1}^{n''}\vphi_0(\xi\tbe_{j,n''})\Big)\\
&=\E\Big(\exp\Big\{i\chi\xi\sum_{j=1}^{n''}\tbe_{j,n''}-k_1|\xi|\sum_{j=1}^{n''}\tbe_{j,n''}-k_1|\xi|\sum_{j=1}^{n''}\tbe_{j,n''}\psi(\xi\tbe_{j,n''})\Big\}\Big).
\end{split}
\]
Since $\psi(\xi)=o(1)$ as $|\xi|\to 0$, for every $\veps>0$ there is a positive $\delta$ such that $|\psi(\xi)|<\veps$ whenever $|\xi|<\delta$. Now, for every $\xi$ in $\R$, in view of the aforesaid property of the maximum of the $\tbe$'s, one can determine the smallest integer $\bar{n}=\bar{n}(\xi,\o)$ such that $|\xi|\tbe_{j,n''}(\o)<\delta$ holds for every $n''\geq\bar{n}$ and $j=1,\dots,n''$, with the exception of a set of points $\o$ of $\CP$-probability $0$. For such $n''$ and $j$, 
\[
\Big|-k_1|\xi|\sum_{j=1}^{n''}\tbe_{j,n''}\psi(\xi\tbe_{j,n''})\Big|\leq k_1|\xi|\veps \sum_{j=1}^{n''}\tbe_{j,n''}
\]
and, then,
\[
\lim_{n''\to+\infty}\Big|-k_1|\xi|\sum_{j=1}^{n''}\tbe_{j,n''}\psi(\xi\tbe_{j,n''})\Big|=0
\] 
holds with $\CP$-probability $1$. Finally, by the dominated convergence theorem, $\lim_{n''\to+\infty}\vphi_{n''}(\xi)=\E(\exp\{i\chi\xi M^{(1)}_\infty-k_1|\xi|M^{(1)}_\infty\})$ completing the proof since the limit is independent of $(n')$.

\subsection{Proof of {\rm Theorem \ref{Th2}}} \label{sec:proofTh2}
As for sufficiency, one can refer to Theorem \ref{ThSuff}. As for necessity, arguing as at the beginning of Section \ref{sec:proofTh1}, one has
\[
\nu(\co)[x_0,+\infty)=\lim_{n\to+\infty}\sum_{j=1}^{\cnun(\co)}\Big[1-F_0\Big(\dfrac{x_0}{\cben_{j,\cnun(\co)}(\co)}\Big)\Big]
\]
where $x_0\in\CI^c\cap(1,+\infty)$ and $\co\in\cO_{x_0}$, $\CI$ is the countable set specified by the application of Lemma \ref{lemma1} to the restriction of $\nu$ to $[1,+\infty)$. Now, letting
\[
I^+:=\liminf_{x\to+\infty}x^\a(1-F_0(x))\qquad\text{and}\qquad S^+:=\limsup_{x\to+\infty}x^\a(1-F_0(x))
\]
and arguing as in the proof of Theorem \ref{Th1}, one concludes that $I^+=S^+<+\infty$. Analogously, one has
\[
\nu(\co)(-\infty,-x'_0]=\lim_{n\to+\infty}\sum_{j=1}^{\cnun(\co)}F_0\Big(-\dfrac{x'_0}{\cben_{j,\cnun(\co)}(\co)}\Big)
\]
where $x'_0\in\CI^c\cap(1,+\infty)$ and  $\co\in \cO_{x'_0}$, $\CI$ is the countable set obtained by the application of Lemma \ref{lemma1} to the measure $\bar{\nu}_1$ on $\CB([1,+\infty))$ defined by $\bar{\nu}_1[x,+\infty):=\nu(-\infty,-x]$ for every $x>1$. Then, resorting once again to the argument developed in the proof of Theorem \ref{Th1}, 
\[
\liminf_{x\to+\infty}x^\a F_0(-x)=\limsup_{x\to+\infty}x^\a F_0(-x)<+\infty.
\]
Thus, weak convergence of $\mu_t$ implies \eqref{NDA}. At this stage, the rest of the argument is splitted into four points, depending on the value assumed by $\a$.\newline

If $\a$ belongs to $(0,1)$, no further consideration is needed.\newline

Passing to the case of $\a=1$, one has to prove that $c_1=c_2$ under the assumption that \eqref{NDA} is in force. Resorting to (16.38) in \cite{FristedtGray} and to the Skorokhod representation (in fact, $\mu_t$ converges weakly),
\[
\begin{split}
&\lim_{n\to+\infty}\sum_{j=1}^{\cnun(\co)}\int_{\R}\Big(-\I_{(-\infty,-1]}(x)+x\I_{(-1,1]}(x)\\
&\qquad\qquad\qquad\qquad\qquad\qquad\qquad\qquad\qquad+\I_{(1,+\infty)}(x)\Big)dF_0\Big(\dfrac{x}{\cben_{j,\cnun(\co)}(\co)}\Big)
\end{split}
\]
exists and is finite for every $\co$ in $\cO$. Denoting it by $\eta(\co)$, and using the change of variable $y=x/\cben_{j,\cnun(\co)}(\co)$ one gets
\begin{equation}\label{tclnu}
\begin{split}
\eta(\co)&=\lim_{n\to+\infty}\sum_{j=1}^{\cnun(\co)}\Big[1-F_0\Big(\dfrac{1}{\cben_{j,\cnun(\co)}(\co)}\Big)-F_0\Big(-\dfrac{1}{\cben_{j,\cnun(\co)}(\co)}\Big)\\
&\qquad+\cben_{j,\cnun(\co)}(\co)\int_{-1/\cben_{j,\cnun(\co)}(\co)}^{1/\cben_{j,\cnun(\co)}(\co)}y dF_0(dy)\Big]\\
&= \lim_{n\to+\infty}\sum_{j=1}^{\cnun(\co)}\cben_{j,\cnun(\co)}(\co)\int_0^{1/\cben_{j,\cnun(\co)}(\co)} (1-F_0(x)-F_0(-x))dx
\end{split}
\end{equation}
where the last equality follows from integration by parts. One proceeds to prove that 
\begin{equation}\label{Lim1A=1}
\lim_{R\to+\infty}\int_{(0,R)}(1-F_0(x)-F_0(-x))dx=+\infty
\end{equation}
implies that $\mu_t$ does not converge. Indeed, assuming that \eqref{Lim1A=1} is in force, for every positive $M$ there exists $\overline{R}$ such that $\int_{(0,R)}(1-F_0(x)-F_0(-x))dx\geq M$ for every $R\geq\overline{R}$. Moreover, since $\max_{j=1,\dots,\cnun(\co)}\cben_{j,\cnun(\co)}(\co)\rightarrow0$, as $n\to+\infty$, there exists $\bar{n}=\bar{n}(\co,\overline{R})$ such that $1/\cben_{j,\cnun(\co)}(\co)\geq \overline{R}$ for every $n\geq\bar{n}$ and for every $j=1,\dots,\cnun(\co)$. Thus, putting $\widehat{M}^{(1)}_\infty(\co)=\lim_{n\to+\infty}$ $\sum_{j=1}^{\cnun(\co)}\cben_{j,\cnun(\co)}(\co)$, one has $\widehat{\E}(\widehat{M}^{(1)}_\infty)=1$ and, then, $\cP\{\widehat{M}^{(1)}_\infty>0\}>0$. Moreover, $\eta(\co)\geq M\cdot \widehat{M}^{(1)}_\infty(\co)$ holds true, in view of \eqref{tclnu}, for every $\co$ and $M>0$, yielding a contradiction when $\co\in\{\widehat{M}^{(1)}_\infty>0\}$. Analogously, one proves that
\begin{equation}\label{Lim2A=1}
\lim_{R\to+\infty}\int_{(0,R)}(1-F_0(x)-F_0(-x))dx=-\infty
\end{equation}
implies that $\mu_t$ does not converge. Finally, one proves that, if $c_1\neq c_2$, then either \eqref{Lim1A=1} or \eqref{Lim2A=1} occurs. Indeed, if $c_1<c_2$, taking $\veps$ in $(0,(c_2-c_1)/2)$, since \eqref{NDA} is in force, there is $\bar{x}>0$ such that 
\[
\dfrac{c_2-c_1-2\veps}{x}\leq 1-F_0(x)-F_0(-x)\leq \dfrac{c_2-c_1+2\veps}{x}
\]
for every $x\geq \bar{x}$, and
\[
\begin{split}
&\int_{(0,R)}(1-F_0(x)-F_0(-x))dx\\
&\qquad=\Big(\int_{(0,\bar{x})}+\int_{(\bar{x},R)}\Big)(1-F_0(x)-F_0(-x))dx
\end{split}
\]
\[
\begin{split}
&\qquad\geq \int_{(0,\bar{x})}(1-F_0(x)-F_0(-x))dx+\int_{(\bar{x},R)}\dfrac{c_2-c_1-2\veps}{x}dx\\
&\qquad=\int_{(0,\bar{x})}(1-F_0(x)-F_0(-x))dx+(c_2-c_1-2\veps)\log\dfrac{R}{\bar{x}}
\end{split}
\]
which goes to $+\infty$ as $R\to+\infty$. Analogously, one proves that $c_1>c_2$ entails \eqref{Lim2A=1}. Combination of these facts with the inconsistency between \eqref{Lim1A=1}-\eqref{Lim2A=1} and weak convergence of $\mu_t$ entails $c_1=c_2$.\newline

Now, the case of $\a$ in $(1,2)$ is taken into consideration. Condition \eqref{NDA} implies that $m_{0,1}:=\int_{\R}x\mu_0(dx)$ is finite. The former summand in the RHS of
\[
V_t=\sum_{j=1}^{\tnu_t}(X_j-m_{0,1})\tbe_{j,\tnu_t}+m_{0,1}\sum_{j=1}^{\tnu_t}\tbe_{j,\tnu_t}
\]
converges in distribution, as $t\to+\infty$, in view of Theorem \ref{ThSuff}. As to the latter, one notes that
\[
\sum_{j=1}^{n}\tbe_{j,n}=\sum_{j=1}^{n}\tbe^\a_{j,n}\dfrac{1}{\tbe^{\a-1}_{j,n}}\geq\dfrac{1}{(\max_{j=1,\dots,n}\tbe_{j,n})^{\a-1}}\sum_{j=1}^{n}\tbe^\a_{j,n}
\]
goes in probability to $+\infty$ as $n\to+\infty$, since $\sum_{j=1}^{n}\tbe^\a_{j,n}$ converges almost surely to the random variable $M^{(\a)}_\infty$, satisfying $\CP\{M^{(\a)}_\infty>0\}>0$, and $\max_{j=1,\dots,n}\tbe_{j,n}$ converges in probability to zero. Then, $m_{0,1}=0$.\newline

Finally, if $\a=2$, in view of the previous argument, \eqref{NDA} holds and $m_{0,1}=0$. From the Skorokhod representation combined with (16.37) in \cite{FristedtGray}, there exists $\sigma^2\colon \cO\rightarrow \R^+$ such that, for every $\co$ in $\cO$,
\begin{equation}\label{TCLsigma}
\begin{split}
\sigma^2(\co)&=\lim_{\veps\searrow0}\limsup_{n\to+\infty}\sum_{j=1}^{\cnun(\co)}\Big[\int_{[-\veps,\veps]}x^2dF_0\Big(\dfrac{x}{\cben_{j,\cnun(\co)}(\co)}\Big)\\
&\qquad\qquad\qquad\qquad\qquad\qquad-\Big(\int_{[-\veps,\veps]}x dF_0\Big(\dfrac{x}{\cben_{j,\cnun(\co)}(\co)}\Big)\Big)^2\Big]\\
&=\lim_{\veps\searrow0}\liminf_{n\to+\infty}\sum_{j=1}^{\cnun(\co)}\Big[\int_{[-\veps,\veps]}x^2dF_0\Big(\dfrac{x}{\cben_{j,\cnun(\co)}(\co)}\Big)\\
&\qquad\qquad\qquad\qquad\qquad\qquad-\Big(\int_{[-\veps,\veps]}x dF_0\Big(\dfrac{x}{\cben_{j,\cnun(\co)}(\co)}\Big)\Big)^2\Big].
\end{split}
\end{equation}
With the change of variable $y=x/\cben_{j,\cnun(\co)}(\co)$, 
\[
\begin{split}
&\liminf_{n\to+\infty}\sum_{j=1}^{\cnun(\co)}\Big[\int_{[-\veps,\veps]}x^2dF_0\Big(\dfrac{x}{\cben_{j,\cnun(\co)}(\co)}\Big)\\
&\qquad\qquad\qquad\qquad\qquad\qquad-\Big(\int_{[-\veps,\veps]}x dF_0\Big(\dfrac{x}{\cben_{j,\cnun(\co)}(\co)}\Big)\Big)^2\Big]\\
&\quad=\liminf_{n\to+\infty}\sum_{j=1}^{\cnun(\co)}\Big(\cben_{j,\cnun(\co)}(\co)\Big)^2\Big[\int_{[-\veps/\cben_{j,\cnun(\co)}(\co),\veps/\cben_{j,\cnun(\co)}(\co)]}y^2dF_0(y)\\
&\qquad-\Big(\int_{[-\veps/\cben_{j,\cnun(\co)}(\co),\veps/\cben_{j,\cnun(\co)}(\co)]}y dF_0(y)\Big)^2\Big]\\
&\quad\geq \liminf_{n\to+\infty}\sum_{j=1}^{\cnun(\co)}\Big(\cben_{j,\cnun(\co)}(\co)\Big)^2\int_{[-\veps/\cben_{j,\cnun(\co)}(\co),\veps/\cben_{j,\cnun(\co)}(\co)]}y^2dF_0(y)\\
&\qquad-\limsup_{n\to+\infty}\sum_{j=1}^{\cnun(\co)}\Big(\cben_{j,\cnun(\co)}(\co)\Big)^2\Big(\int_{[-\veps/\cben_{j,\cnun(\co)}(\co),\veps/\cben_{j,\cnun(\co)}(\co)]}y dF_0(y)\Big)^2.
\end{split}
\]
As for the latter summand of the RHS, since $m_{0,1}=0$, for every $\delta>0$ there exists an $\bar{R}$ such that $\Big|\int_{[-R,R]}x dF_0(x)\Big|<\delta$ whenever $R\geq\bar{R}$. Let $\bar{n}=\bar{n}(\co,\bar{R})$ be a strictly positive integer such that $\cben_{j,\cnun(\co)}(\co)\leq \veps/\bar{R}$ for every $n\geq\bar{n}$ and for every $j=1,\dots,\cnun(\co)$. For such $j$'s and $n$'s one has
\[
-\delta<\int_{[-\veps/\cben_{j,\cnun(\co)}(\co),\veps/\cben_{j,\cnun(\co)}(\co)]}y dF_0(y)<\delta
\]
and
\[
\begin{split}
&\sum_{j=1}^{\cnun(\co)}\Big(\cben_{j,\cnun(\co)}(\co)\Big)^2\Big(\int_{[-\veps/\cben_{j,\cnun(\co)}(\co),\veps/\cben_{j,\cnun(\co)}(\co)]}y dF_0(y)\Big)^2\\
&\qquad\leq \delta^2 \sum_{j=1}^{\cnun(\co)}\Big(\cben_{j,\cnun(\co)}(\co)\Big)^2.
\end{split}
\]
Thus, as the last inequality holds for every $\delta>0$ and $\sum_{j=1}^{\cnun(\co)}\Big(\cben_{j,\cnun(\co)}(\co)\Big)^2$ converges, as $n\to+\infty$, to a positive $\widehat{M}^{(2)}_\infty(\co)$, 
\[
\limsup_{n\to+\infty}\sum_{j=1}^{\cnun(\co)}\Big(\cben_{j,\cnun(\co)}(\co)\Big)^2\Big(\int_{[-\veps/\cben_{j,\cnun(\co)}(\co),\veps/\cben_{j,\cnun(\co)}(\co)]}y dF_0(y)\Big)^2=0.
\]
Finally, one proves that $m_{0,2}:=\int_{\R}x^2 dF_0(x)$ is finite. In fact, if $m_{0,2}=+\infty$, for every $M>0$ there is $\bar{R}>0$ such that $\int_{[-R,R]}x^2dF_0(x)\geq M$ holds for every $R\geq\bar{R}$. Then, an application of the same argument as in the previous step gives
\[
\begin{split}
&\liminf_{n\to+\infty}\sum_{j=1}^{\cnun(\co)}\Big(\cben_{j,\cnun(\co)}(\co)\Big)^2\int_{[-\veps/\cben_{j,\cnun(\co)}(\co),\veps/\cben_{j,\cnun(\co)}(\co)]}y^2dF_0(y)\\
&\qquad\geq M\cdot \widehat{M}^{(2)}_\infty(\co)
\end{split}
\]
which turns out to be an apparent contradiction since $M$ is arbitrary and $\sigma^2(\co)$ is finite.

\subsection{Proof of {\rm Theorem \ref{Tha>2}}}\label{sec:proofTha>2}
The argument to prove Theorem \ref{Th1} can be plainly extended to the case of $\a>2$ to state that $\lim_{x\to+\infty}x^\a(1-F^*_0(x))$ exists and is finite, which implies $m_{0,2}<+\infty$. An integration by parts followed by the change of variable $y=x/|\cben_{j,\cnun(\co)}(\co)|$ transforms the sum in the RHS of \eqref{TCLsigma} into
\[
\begin{split}
&\sum_{j=1}^{\cnun(\co)}\Big|\cben_{j,\cnun(\co)}(\co)\Big|^2\Big[\int_{[-\veps/|\cben_{j,\cnun(\co)}(\co)|,\veps/|\cben_{j,\cnun(\co)}(\co)|]}y^2dF_0(y)\\
&\qquad\qquad\qquad\qquad\qquad-\Big(\int_{[-\veps/|\cben_{j,\cnun(\co)}(\co)|,\veps/|\cben_{j,\cnun(\co)}(\co)|]}ydF_0(y)\Big)^2\Big].
\end{split}
\]
For every $\delta>0$ there is $\bar{R}>0$ such that $m_{0,i}-\delta<\int_{[-R,R]}x^i dF_0(x)<m_{0,i}+\delta$ holds for every $R\geq\bar{R}$ and $i=1,2$. Moreover, let $\bar{n}=\bar{n}(\co,\bar{R})$ be a strictly positive integer such that $\veps/|\cben_{j,\cnun(\co)}(\co)|>\bar{R}$ for every $n\geq\bar{n}$ and $j=1,\dots,\cnun(\co)$. Then,
\[
\begin{split}
&\sum_{j=1}^{\cnun(\co)}\Big|\cben_{j,\cnun(\co)}(\co)\Big|^2\Big[\int_{[-\veps/|\cben_{j,\cnun(\co)}(\co)|,\veps/|\cben_{j,\cnun(\co)}(\co)|]}y^2dF_0(y)\\
&\qquad\qquad\qquad\qquad\qquad-\Big(\int_{[-\veps/|\cben_{j,\cnun(\co)}(\co)|,\veps/|\cben_{j,\cnun(\co)}(\co)|]}ydF_0(y)\Big)^2\Big]
\end{split}
\]
\[
\begin{split}
&\qquad\geq \sum_{j=1}^{\cnun(\co)}\Big|\cben_{j,\cnun(\co)}(\co)\Big|^2\Big[m_{0,2}-\delta-\Big(m_{0,1}+\delta\Big)^2\Big].
\end{split}
\]
Taking $\liminf_{n\to+\infty}$ in both sides of the above inequality one gets   
\[
\sigma^2(\co)\geq \liminf_{n\to+\infty}\sum_{j=1}^{\cnun(\co)}\Big|\cben_{j,\cnun(\co)}(\co)\Big|^2 (m_{0,2}-m_{0,1}^2).
\]
Since $\a>2$, one has 
\[
\sum_{j=1}^{\cnun(\co)}\Big|\cben_{j,\cnun(\co)}(\co)\Big|^2\geq \dfrac{\sum_{j=1}^{\cnun(\co)}\Big|\cben_{j,\cnun(\co)}(\co)\Big|^\a}{\Big(\max_{j=1,\dots,\cnun(\co)}|\cben_{j,\cnun(\co)}(\co)| \Big)^{\a-2}}
\]
which goes to $+\infty$ for every $\co$. Now, because of the finiteness of $\sigma^2(\co)$, one must have $m_{0,2}-m_{0,1}^2=0$, i.e. $\mu_0$ is the point mass at some $x_0$ in $\R$. Conversely, if $\mu_0=\delta_{x_0}$, 
\[
\begin{split}
&\sum_{j=1}^{n}\Big|\tbe_{j,n}(\o)\Big|^2\Big[\int_{[-\veps/|\tbe_{j,n}(\o)|,\veps/|\tbe_{j,n}(\o)|]}y^2dF_0(y)\\
&\qquad\qquad\qquad\qquad\qquad-\Big(\int_{[-\veps/|\tbe_{j,n}(\o)|,\veps/|\tbe_{j,n}(\o)|]}ydF_0(y)\Big)^2\Big]\\
&\qquad=\sum_{j=1}^{n}\Big|\tbe_{j,n}(\o)\Big|^2 x_0^2\Big[\I_{[-\veps/|\tbe_{j,n}(\o)|,\veps/|\tbe_{j,n}(\o)|]}(x_0)\\
&\qquad\qquad\qquad\qquad\qquad-\I^2_{[-\veps/|\tbe_{j,n}(\o)|,\veps/|\tbe_{j,n}(\o)|]}(x_0)\Big]\\
&\qquad=0
\end{split}
\]
which entails  
\[
\begin{split}
\sigma^2(\o)&=\lim_{\veps\searrow0}\limsup_{n\to+\infty}\sum_{j=1}^{n}\Big[\int_{[-\veps,\veps]}x^2dF_0\Big(\dfrac{x}{\tbe_{j,n}(\o)}\Big)\\
&\qquad\qquad\qquad\qquad\qquad\qquad-\Big(\int_{[-\veps,\veps]}x dF_0\Big(\dfrac{x}{\tbe_{j,n}(\o)}\Big)\Big)^2\Big]\\
&=\lim_{\veps\searrow0}\liminf_{n\to+\infty}\sum_{j=1}^{n}\Big[\int_{[-\veps,\veps]}x^2dF_0\Big(\dfrac{x}{\tbe_{j,n}(\o)}\Big)\\
&\qquad\qquad\qquad\qquad\qquad\qquad-\Big(\int_{[-\veps,\veps]}x dF_0\Big(\dfrac{x}{\tbe_{j,n}(\o)}\Big)\Big)^2\Big]\\
&=0
\end{split}
\]
for every $\o$ in $\O$. To complete the proof that $\mu_\infty$ degenerates at some $x_1$, one can resort to the central limit theorem (cf., e.g., (16.36) in \cite{FristedtGray}) according to which one has to check that 
\[
\nu(I)=\lim_{n\to+\infty}\sum_{j=1}^n \delta_{x_o\tbe_{j,n}}(I)=0
\]  
holds for every $I\in\{(-\infty,a],$ $[b,+\infty):\;a<0,b>0\}$, which can be plainly verified by recalling that $\max_{j=1,\dots,n}\tbe_{j,n}$ goes to zero in probability. 

It remains to characterize the point $x_1$ at which $\mu_\infty$ degenerates. From $Q^+(\mu_\infty)=\mu_\infty=\delta_{x_1}$ one has
\[
e^{i\xi x_1}=\widehat{Q^+}(\delta_{x_1})(\xi)=\E\Big(\hat{\delta}_{x_1}(\xi\tL_1)\hat{\delta}_{x_1}(\xi\tR_1)\Big)=\E\Big(e^{i\xi(\tL_1+\tR_1)x_1}\Big)
\]
which implies that $\tL_1+\tR_1=1$ almost surely when $x_1\neq0$. Moreover, in this case, $V_t$ turns out to be equal to $x_0$ with probability one since all the $X_j$'s are degenerate at $x_0$, and the condition $\CP\{\tL_1+\tR_1=1\}=1$, combined with \eqref{betaricorsivo}, entails $\sum_{j=1}^{\tnu_t}\tbe_{j,\tnu_t}=1$ almost surely. Then, $x_0=x_1$. Conversely, this very same argument proves that $\CP\{\tL_1+\tR_1=1\}=1$ imply that $x_1=x_0$. Finally, in order that the solution $\mu_t$ converge weakly to the point mass at zero, it is necessary and sufficient that $V_t=x_0 \sum_{j=1}^{\tnu_t}\tbe_{j,\tnu_t}$ converge in law to zero, which happens when (at least) one of the conditions $(ii_1)$, $(ii_2)$ is verified.

\subsection{Proof of {\rm Theorem \ref{Th4}}}\label{sec:proofTh4}
It is well-known that the Fourier-Stieltjes transform of the solution of \eqref{eq} has the \textit{Wild series representation}  
\[
\vphi(t,\xi)=\sum_{n\geq1}e^{-t}(1-e^{-t})^{n-1}\hat{q}_n(\xi,\vphi_0)
\]
where $\hat{q}_1(\xi;\vphi):=\vphi_0(\xi)$ and, for every $n\geq2$, 
\[
\hat{q}_n(\xi;\vphi):=\dfrac{1}{n-1}\sum_{j=1}^{n-1}\E[\hat{q}_{j}(\tL_1\xi;\vphi)\hat{q}_{n-j}(\tR_1\xi;\vphi)].
\]
In fact, if $\Re z$ [$\Im z$] denotes the real [imaginary] part of a complex number $z$, it will be shown that
\begin{equation}\label{PhiReale}
\vphi(t,\xi)=e^{-t}\sum_{n\geq1}(1-e^{-t})^{n-1}\hat{q}_n(\xi,\Re \vphi_0)+i\Im \vphi_0(\xi)e^{-t}
\end{equation}
and, thus, the study of the limiting behaviour of $\vphi(t,\cdot)$ can be done through the study of the Cauchy problem associated with \eqref{eq} with initial datum given by $\vphi^*_0(\cdot):=\Re \vphi_0(\cdot)$, that is the Fourier-Stieltjes transform of $\mu^*_0$. It is worth noticing that the corresponding p.d.f. $F^*_0$ satisfies $F^*_0(-x)=1-F^*_0(x)$ on the set of the continuity points. Hence, assuming \eqref{PhiReale}, one can think of $\tL_1$ and $\tR_1$ as positive random variables, without loss of generality, so that the present theorem appears to be a part of Theorem \ref{Th2}. In point of fact, it remains to prove \eqref{PhiReale}, which is implied by
\begin{equation}\label{q}
\hat{q}_n(\xi;\vphi_0)=\hat{q}_n(\xi;\Re\vphi_0)\text{ for every }n\geq2.
\end{equation}
Proceeding by mathematical induction, one first proves \eqref{q} when $n=2$. Write
\[
\begin{split}
\hat{q}_2(\xi;\vphi)&=\E[\Re\vphi_0(\tL_1\xi)\Re\vphi_0(\tR_1\xi)]+i\E[\Re\vphi_0(\tL_1\xi)\Im\vphi_0(\tR_1\xi)]\\
&\qquad+i\E[\Im\vphi_0(\tL_1\xi)\Re\vphi_0(\tR_1\xi)]-\E[\Im\vphi_0(\tL_1\xi)\Im\vphi_0(\tR_1\xi)]\\
&=: A_1+i A_2+i A_3-A_4.
\end{split}
\]
Now,
\[
\begin{split}
A_2&= \E\Big[\Re\vphi_0(\tL_1\xi)\Im\vphi_0(\tR_1\xi)\I_{\{\tR_1>0\}}\Big]+ \E\Big[\Re\vphi_0(\tL_1\xi)\Im\vphi_0(-(-\tR_1)\xi)\I_{\{-\tR_1>0\}}\Big]\\
&= \E\Big[\Re\vphi_0(\tL_1\xi)\Im\vphi_0(\tR_1\xi)\I_{\{\tR_1>0\}}\Big]- \E\Big[\Re\vphi_0(\tL'_1\xi)\Im\vphi_0(\tR'_1\xi)\I_{\{\tR'_1>0\}}\Big]\\
&\qquad\qquad\qquad\qquad\qquad\qquad\qquad\qquad\text{(where $(\tL'_1,\tR'_1):=(-\tL_1,-\tR_1)$)}\\
&=0
\end{split}
\]
the last equality being a consequence of the fact that $(\tL'_1,\tR'_1)$ and $(\tL_1,\tR_1)$ have the same distribution because of the invariance w.r.t. $(\pi/2)$-rotations. In the same way, one can prove that $A_3=0$. As for $A_4$, recalling the above definitions of $(\tL'_1,\tR'_1)$ and putting $(\tL''_1,\tR''_1):=(\tR_1,-\tL_1)$ $-$ which, in view of invariance w.r.t. $(\pi/2)$-rotation, is distributed like $(\tL_1,\tR_1)$ $-$ one has
\[
\begin{split}
A_4&= \E\Big[\Im\vphi_0(\tL_1\xi)\Im\vphi_0(\tR_1\xi)\I_{\{\tR_1>0\}}\Big]+ \E\Big[\Im\vphi_0(-\tL'_1\xi)\Im\vphi_0(-\tR'_1\xi)\I_{\{\tR'_1>0\}}\Big]\\
&= 2\E\Big[\Im\vphi_0(\tL_1\xi)\Im\vphi_0(\tR_1\xi)\I_{\{\tR_1>0\}}\Big]\\
&= 2\E\Big[\Im\vphi_0(\tL_1\xi)\Im\vphi_0(\tR_1\xi)\I_{\{\tR_1>0,\tL_1>0\}}\Big]\\
&\qquad\qquad\qquad\qquad\qquad\qquad+2\E\Big[\Im\vphi_0(-\tR''_1\xi)\Im\vphi_0(\tL''_1\xi)\I_{\{\tR''_1>0,\tL''_1>0\}}\Big]\\
&=0.
\end{split}
\]
Verified that \eqref{q} holds true for $n=2$, one assumes its validity for every $n\leq m-1$ ($m\geq3$) and proves it for $n=m$. From the definition of $\hat{q}_n$ in conjunction with the inductive hypothesis,
\[
\begin{split}
\hat{q}_m(\xi,\vphi_0)&=\dfrac{1}{m-1}\Big(\E\Big[\hat{q}_{m-1}(\tL_1\xi,\Re\vphi_0)\vphi_0(\tR_1\xi)\\
&\qquad+\sum_{j=2}^{m-2}\hat{q}_{m-j}(\tL_1\xi;\Re\vphi)\hat{q}_j(\tR_1\xi;\Re\vphi_0)+\vphi_0(\tL_1\xi)\hat{q}_{m-1}(\tR_1\xi;\Re\vphi_0)\Big]\Big)\\
&=\hat{q}_m(\xi;\Re\vphi_0)+\dfrac{i}{m-1}\Big(\E\Big[\hat{q}_{m-1}(\tL_1\xi;\Re\vphi_0)\Im\vphi_0(\tR_1\xi)\\
&\qquad+\Im\vphi_0(\tL_1\xi)\hat{q}_{m-1}(\tR_1\xi;\Re\vphi_0)\Big]\Big).
\end{split}
\]
For every $k\geq1$, $\xi\mapsto\hat{q}_k(\xi;\Re\vphi_0)$ is an even function and then, arguing as for $A_2$ and $A_3$, one gets $\E\Big[\hat{q}_{m-1}(\tL_1\xi;\Re\vphi_0)\Im\vphi_0(\tR_1\xi)\Big]=\E\Big[\hat{q}_{m-1}(\tR_1\xi;\Re\vphi_0)$ $\Im\vphi_0(\tL_1\xi)\Big]=0$ and hence
$\hat{q}_m(\xi;\vphi_0)=\hat{q}_m(\xi;\Re\vphi_0)$. This completes the inductive argument and the proof of \eqref{q} and hence \eqref{PhiReale}.

\section*{Acknowledgements}
We are grateful to the referee for constructive comments and suggestions that have much improved the exposition.

\appendix 
{\normalsize
\section{Proofs of the lemmas}\label{sec:Appendix1}
This Appendix contains the proofs of Lemmas \ref{lemma1} and \ref{lemma2} which are crucial for the arguments developed in Section \ref{sec:proof}.
\subsection{Proof of Lemma \ref{lemma1} }\label{prooflemma1}
Consider the following p.d.f.'s
\[
\begin{array}{ll}
F_t(x):=\P\{\nu[t,+\infty)\leq x\}, & F_{t,k}(x):=\P\Big(\nu[t,+\infty)\leq x\Big|\nu[1,+\infty)\leq k\Big)\\
G_t(x):=\P\{\nu(t,+\infty)\leq x\}, & G_{t,k}(x):=\P\Big(\nu(t,+\infty)\leq x\Big|\nu[1,+\infty)\leq k\Big)
\end{array}
\]
at each $x$ in $\R$, for every $t\geq1$ and for every $k$ in $\N$. Since $\nu[1,+\infty)$ is almost surely finite, 
\[
\lim_{k\to+\infty}\P\{\nu[1,+\infty)\leq k\}=1
\]
and then
\[
\lim_{k\to+\infty}F_{t,k}(x)=F_t(x)\quad\text{and}\quad \lim_{k\to+\infty}G_{t,k}(x)=G_t(x)
\]
for every $t\geq1$ and $x$ in $\R$. Moreover, $G_{t,k}(x)\geq F_{t,k}(x)$ for every $t\geq1$, $x$ in $\R$ and $k$ in $\N$. Now, fix $k$ and suppose that there exists an uncountable subset $H$ of $(1,+\infty)$ such that, for every $t$ in $H$, 
\begin{equation}\label{pdfdiverse}
G_{t,k}\gvertneqq F_{t,k}
\end{equation}
which means that for every $t$ in $H$ there exists $x_{t,k}\geq1$ such that $G_{t,k}(x_{t,k})>F_{t,k}(x_{t,k})$. The p.d.f.'s $F_{t,k}$ and $G_{t,k}$ are right-continuous and then for every $t$ in $H$ there exists a proper interval $\Delta_{t,k}$ containing $x_{t,k}$ such that
\[
G_{t,k}(x)>F_{t,k}(x)\qquad\text{for every $x$ in $\Delta_{t,k}$}.
\]
One proves that the intersection of any uncountable family of elements of $(\Delta_{t,k})_{t\in H}$ is empty. Suppose, for the moment, that there is an uncountable subset $J$ of $H$ such that $\bigcap_{t\in J}\Delta_{t,k}$ is non-empty; it will be shown that this leads to a contradiction. If $\bar{x}$ is an element of such an intersection, then $G_{t,k}(\bar{x})>F_{t,k}(\bar{x})$ for every $t$ in $J$. Since $G_{t,k}\leq F_{s,k}$ whenever $s>t$, the class $\Big((F_{t,k}(\bar{x}),G_{t,k}(\bar{x})]\Big)_{t\in J}$ consists of pairwise disjoint proper intervals contained in $[0,1]$, which contradicts, recalling that $J$ is uncountable, the countability of the rationals. Verified that $\bigcap_{t\in J}\Delta_{t,k}=\emptyset$ for every uncountable $J\subset H$, one shows that \eqref{pdfdiverse} may be satisfied only on countable sets of $t$'s. In the beginning, one notes that, since 
\[
\begin{array}{ll}
F_{t,k}(x)=G_{t,k}(x)=0 & \text{($x<0$)}\\
F_{t,k}(x)=G_{t,k}(x)=1 & \text{($x>k$)}
\end{array}
\]
for every $t$ in $H$, all the elements of the family $(\Delta_{t,k})_{t\in H}$ are proper sub-intervals of $[0,k]$ such that $\bigcap_{t\in J}\Delta_{t,k}=\emptyset$ for every uncountable $J\subset H$. This last statement implies that with each rational $q$ in $[0,k]$ one can associate a countable (possibly empty) set $H_q\subset H$ such that $q\in\Delta_{t,k}$ for every $t$ in $H_q$. Then, the family $\{\Delta_{t,k}:\;t\in\bigcup_{q\in \Q\cap[0,k]}H_q\}$ is countable and, of course, it is included in $(\Delta_{t,k})_{t\in H}$. So, to complete the argument, it is enough to show that these families coincide. In point of fact, if there exists some $t \in H\setminus \bigcup_{q\in \Q\cap[0,k]}H_q$, then $\Delta_{t,k}\cap \Q=\emptyset$, a patent contradiction. Whence, one can say there is a countable subset $\CI_k$ of $(1,+\infty)$ such that, for every $t$ in $\CI_k^c\cap(1,+\infty)$, $G_{t,k}\equiv F_{t,k}$. Denoting the countable set $\bigcup_{k\geq1}\CI_k$ by $\CI$, the identity $G_{t,k}\equiv F_{t,k}$ holds for every $t$ in $\CI^c\cap(1,+\infty)$ and for every $k$. Moreover, for all of these $t$'s,
\[
F_t(x)=\lim_{k\to+\infty}F_{t,k}(x)=\lim_{k\to+\infty}G_{t,k}(x)=G_t(x)
\]
obtains for every $x$ in $\R$, which is tantamount to stating that $\nu[t,+\infty)$ and $\nu(t,+\infty)$ are equally distributed. Then, since two positive and equally distributed random numbers $X$ and $Y$ such that $X\geq Y$ must coincide almost surely, one concludes that $\nu[t,+\infty)=\nu(t,+\infty)$ almost surely, which amounts to $\nu\{t\}=0$ almost surely. 

\subsection{Proof of Lemma \ref{lemma2}}\label{prooflemma2}
Following the argument used in the proof of Proposition 1 in \cite{BonPerReg}, there is a subset $\cO'$ of $\cO$, $\cP(\cO')=1$, such that the recursive relation \eqref{ricorrenzabeta} holds true at each point of $\cO'$. Then, without altering the distribution of $\cW_n$, one can use \eqref{ricorrenzabeta} to redefine the $\cben$'s outside $\cO'$. Let $M$ be the compact space defined by
\[
M:=\overline{\N}^\infty\times\Big(\times_{j\geq1}\overline{\N}_j^\infty\Big)\times\Big(\times_{j\geq1}(\overline{\R}^2_j)^\infty\Big)
\]
where $\overline{\N}_1,\overline{\N}_2,\dots$ are copies of $\overline{\N}:=\{1,2,\dots,+\infty\}$ and $\overline{\R}_1,\overline{\R}_2,\dots$ are copies of $\overline{\R}$. Introduce the mapping $\hat{Y}$ from $\cO$ to $M$ 
\[
\hat{Y}:=\Big( (\cnu^{(n)})_{n\geq1}, (\ci^{(n)})_{n\geq1}, ((\cLn,\cRn))_{n\geq1}\Big)
\]
and put
\[
\begin{split}
&f_k(\hat{Y})\\
&\quad:=\Big((\cnu^{(1)},\ci^{(1)}),\dots,(\cnu^{(k)},\ci^{(k)}), (\hat{L}^{(1)},\hat{R}^{(1)}),(\hat{L}^{(2)},\hat{R}^{(2)}),\dots,(\hat{L}^{(k)},\hat{R}^{(k)})\Big)
\end{split}
\]
$k=1,2,\dots$. Recall that the pair $(\cnu^{(k)},\ci^{(k)})$ is enough to single out a specific McKean tree, say $\ca_k$. The proof aims at the definition of a non-increasing sequence $(A_n)_{n\geq1}$ of non-empty compact subsets of $M$ such that, if $\cY$ belongs to $A_n$, then \eqref{betaLemma}-\eqref{sommaLemma} hold simultaneously for every $n\geq1$. The expression "weight of a leaf" will be used during the proof to designate the value of the $\beta$ associated with that leaf. Since both hypothesis \eqref{ip2} and the thesis of the present lemma are concerned with the absolute value of the $\tL_i$'s, $\tR_i$'s, $\tbe$'s, with a view to simplifying the notation in the various steps of the proof, these random elements will be supposed to be positive.\newline
\textbf{Step 1.} This step shows that a node weighted by $1/c$ (for some $c\geq1$) can be the root node of a tree with a certain number $N$ of leaves in such a way that each of $(N-1)$ of them is weighted by $1/x$ (for some fixed $x>c$) and the remaining one has weight not greater than $1/x$. Thus, define $N:=\lfloor (x/c)^\a\rfloor+\I_{\{ (x/c)^\a\notin \N\}}$ and construct the tree of $N$ leaves in such a way that the depth of the leaf $1$ is $(N-1)$ and the depth of the leaf $j$ ($j=2,\dots,N$) is equal to $(N+1-j)$. This amounts to the tree constructed by taking $i_1=i_2=\dots=i_{N-1}=1$. Moreover, for every $k=1,\dots,N-1$, one sets
\[
R_k:=\dfrac{c}{(x^\a-(k-1)c^\a)^{1/\a}} \quad\text{and}\quad L_k:=(1-R^\a_k)^{1/\a}=\Big(\dfrac{x^\a-k c^\a}{x^\a-(k-1) c^\a}\Big)^{1/\a}.
\]
It is easy to verify that $R_k=c/(x\prod_{j=1}^{k-1} L_j)$, for every $k=1,\dots,N-1$, with the proviso that $\prod_{j=1}^0 L_j=1$. This way, 
\[
\beta_{1,N}=\dfrac{1}{c}\prod_{j=1}^{N-1}L_j,\qquad\beta_{k,N}=\dfrac{1}{c}R_{N-k+1}\prod_{j=1}^{N-k}L_j\quad(k=2,\dots,N)
\]
and, by the definition of $(L_1,R_1),\dots,(L_{N-1},R_{N-1})$,
\[
\beta_{1,N}=\dfrac{(x^\a-(N-1)c^\a)^{1/\a}}{cx},\quad \beta_{k,N}=\dfrac{1}{x}\quad(k=2,\dots,N).
\]
It should be noted that if $(x/c)^\a$ is an integer, then $N=(x/c)^\a$ and $\beta_{1,N}=1/x$, whilst $\beta_{1,N}<1/x$ whenever $(x/c)^\a$ is not an integer: In both cases, $\beta_{1,N}\leq 1/x$.\newline
\textbf{Step 2.} In this step one describes the construction of the sequences $(N_n)_{n\geq1}$ and $(\CR_n)_{n\geq1}$ by a recursive procedure. For $n=1$, by applying Step $1$ with $c=1$ and $x=y_1$, one obtains a tree $a_1$ with $N_1=\lfloor y_1^\a\rfloor+\I_{\{ y_1^\a\notin\N\}}$ leaves such that
\[
\beta_{j,N_1}\leq \dfrac{1}{y_1}\text{ if }j\in\CR_1,\quad\beta_{j,N_1}=\dfrac{1}{y_1}\text{ if }j\notin\CR_1
\]
with: $|\CR_1|=1$ (since $\CR_1=\{1\}$), $\sum_{j=1}^{N_1}\beta_{j,N_1}^\a=1$ and $\sum_{j\in\CR_1}\beta_{j,N_1}^\a<\veps$. Given the tree $a_{n-1}$ ($n\geq2$) with $N_{n-1}$ leaves such that
\[
\beta_{j,N_{n-1}}\leq\dfrac{1}{y_{n-1}}\text{ if }j\in\CR_{n-1},\quad\beta_{j,N_{n-1}}=\dfrac{1}{y_{n-1}}\text{ if }j\notin\CR_{n-1}
\]
and
\[
|\CR_{n-1}|=N_{n-2},\quad \sum_{j=1}^{N_{n-1}}\beta^\a_{j,N_{n-1}}=1,\quad \sum_{j\in\CR_{n-1}}\beta^\a_{j,N_{n-1}}<\veps
\] 
one obtains $a_n$ by applying the construction presented in Step $1$ to each leaf of $a_{n-1}$. More precisely, for each leaf $j$ of $a_{n-1}$, with $j=1,\dots,N_{n-1}$, one implements Step $1$ with $c=1/\beta_{j,N_{n-1}}$ and $x=y_n$, where, passing to a subsequence if needed $-$ but maintaining, in any case, the same symbol $y_n$ $-$ $y_n$ is assumed to be strictly greater than $c$. Thus, the tree appended to leaf $j$ has exactly one leaf with weight not greater than $1/y_n$, and each of the remaining leaves with weight equal to $1/y_n$. Iterating the procedure for $j=1,\dots,N_{n-1}$, one obtains the tree denoted by $a_n$. The symbol $N_n$ stands for the number of the leaves of $a_n$: There are $(N_n-N_{n-1})$ leaves weighted by $1/y_n$ and $N_{n-1}$ leaves with a weight not greater than $1/y_n$. This is equivalent to saying that there exists $\CR_n\subset\{1,\dots,N_{n}\}$ such that $|\CR_n|=N_{n-1}$ and
\[
\beta_{j,N_{n}}\leq\dfrac{1}{y_{n}}\text{ if }j\in\CR_{n},\quad\beta_{j,N_{n}}=\dfrac{1}{y_{n}}\text{ if }j\notin\CR_{n}
\]
where, by construction, $\sum_{j=1}^{N_{n}}\beta_{j,N_{n}}^\a=1$. To conclude with this step, it remains to prove that $\sum_{j\in\CR_{n}}\beta_{j,N_{n}}^\a<\veps$. To this end, it is enough to show that
\begin{equation}\label{N_n}
N_n\leq \sum_{i=1}^n (y_i^\a+1)\qquad\text{for every }n\geq1
\end{equation}
since, if \eqref{N_n} holds, then
\[
\begin{split}
\sum_{j\in\CR_{n}}\beta_{j,N_{n}}^\a&\leq\dfrac{N_{n-1}}{y_n^\a}\leq\dfrac{1}{y_n^\a}\sum_{i=1}^{n-1}(y_i^\a+1) \\
&<\veps\qquad\qquad\qquad\qquad\text{(since $(y_n)_{n\geq1}$ satisfies \eqref{y_n})}.
\end{split}
\]
Coming back to \eqref{N_n}, one proceeds by mathematical induction. From the definition of $N_1$ one gets $N_1\leq y_1^\a+1$, that is the claim for $n=1$. One now supposes that \eqref{N_n} is satisfied for every $n\leq m$. Since, for every $k\geq1$, $\beta_{j,N_k}=1/y_k$ for every $j\notin\CR_k$, $|\CR_k|=N_{k-1}$ and $\sum_{j=1}^{N_k}\beta_{j,N_k}^\a=1$, then 
\[
\dfrac{N_k-N_{k-1}}{y_k^\a}=\sum_{j\notin\CR_k}\beta_{j,N_k}^\a\leq 1
\]
obtains, entailing $N_k-N_{k-1}\leq y_k^\a$. Combination of this with the inductive hypothesis yields
\[
N_{m+1}=N_{m+1}-N_m+N_m\leq y_{m+1}^\a+\sum_{i=1}^m(y_i^\a+1)\leq \sum_{i=1}^{m+1}(y_i^\a+1)
\]
which is \eqref{N_n} for $n=m+1$.\newline
\textbf{Step 3.} After constructing sequences $(N_n)_{n\geq1}$, $(\CR_n)_{n\geq1}$, $(a_n)_{n\geq1}$ according to Step $2$, one now determines suitable proper intervals included in the ranges of the random elements $\cLn_i$'s, $\cRn_i$'s for every $n\geq1$. At each step of the argument one will define the $\cRn_i$'s in such a way to satisfy \eqref{betaLemma}, whilst the $\cLn_i$'s will be assigned residually so that $(\cLn_i,\cRn_i)$ belongs to the support of $\tau$. As far as $n=1$ is concerned, one considers the tree $a_1=(\cnu^{(1)},\ci^{(1)})$. To satisfy \eqref{betaLemma}, one can impose that
\[
\cR^{(1)}_1\in B_1^{(1)}(\delta_1^{(1)}):=\Big[\dfrac{1}{y_1+\delta_1^{(1)}},\dfrac{1}{y_1}\Big]
\]
where the strictly positive $\delta_1^{(1)}$ is determined at a latter time. To specify an interval $C_1^{(1)}(\delta_1^{(1)})$ for $\cL^{(1)}_1$, with a view to \eqref{ip2} one forces $C_1^{(1)}(\delta_1^{(1)})$ to satisfy
\[
\begin{split}
\cL^{(1)}_1\in C_1^{(1)}(\delta_1^{(1)})&:=\Big[\Big(1-\dfrac{1}{y_1^\a}\Big)^{1/\a},\Big(1-\dfrac{1}{(y_1+\delta_1^{(1)})^\a}\Big)^{1/\a}\Big]\\
&=\text{range of }(1-x^\a)^{1/\a}\text{ as }x\text{ varies in }B_1^{(1)}(\delta_1^{(1)}).
\end{split}
\]
To single out an interval $B_2^{(1)}(\delta_1^{(1)})$ for $\cR^{(1)}_2$, with a view to \eqref{betaLemma}, $\cL^{(1)}_1\cdot\cR^{(1)}_2$ must belong to $[1/(y_1+\delta_1^{(1)}),1/y_1]$ for any value of $\cL^{(1)}_1$ in $C_1^{(1)}(\delta_1^{(1)})$, which is guaranteed if
\[
\cR^{(1)}_2\in B_2^{(1)}(\delta_1^{(1)}):=\Big[\dfrac{y_1}{(y_1+\delta_1^{(1)})(y_1^\a-1)^{1/\a}},\dfrac{y_1+\delta_1^{(1)}}{y_1[(y_1+\delta_1^{(1)})^\a-1)]^{1/\a}}\Big]
\]
and this, in turn, allows the following specification
\[
\cL^{(1)}_2\in C_2^{(1)}(\delta_1^{(1)}):=\text{range of }(1-x^\a)^{1/\a}\text{ as }x\text{ varies in } B_2^{(1)}(\delta_1^{(1)}).
\]
The procedure can be iterated to yield intervals $C_1^{(1)}(\delta_1^{(1)}),B_1^{(1)}(\delta_1^{(1)}),\dots,$ $C_{N_1-1}^{(1)}(\delta_1^{(1)}),$ $B_{N_1-1}^{(1)}(\delta_1^{(1)})$ in such a way that \eqref{betaLemma} is met, with $n=1$, whenever
\[
(\cL^{(1)}_k,\cR^{(1)}_k)_{k=1,\dots,N_1-1}\in\times_{k=1}^{N_1-1}\Big(C_k^{(1)}(\delta_1^{(1)})\times B_k^{(1)}(\delta_1^{(1)})\Big)=:I_1(\delta_1^{(1)}).
\]
Note that the validity of this last claim is guaranteed by the way followed so far to construct the above intervals. To complete the construction, it remains to specify admissible values of $\delta_1^{(1)}$ in such a way that \eqref{sommaLemma} holds for $n=1$. One starts by noticing, in the notation of Step 1, that $1=(\sum_{j\in\CR_1}+\sum_{j\notin\CR_1})\beta_{j,N_1}^\a<\veps+(N_1-1)/y_1^\a$, which entails $(N_1-1)/y_1^\a>1-\veps$. Taking $h_1$ and $\eta_1$ such that 
\[
0<h_1<\dfrac{N_1-1}{y_1^\a}-1+\veps \qquad\text{and}\qquad 0<\eta_1<(1+h_1)^{1/N_1}-1
\]
one gets $(\cL^{(1)}_k)^\a+(\cR^{(1)}_k)^\a\leq 1+\eta_1$ for every $k=1,\dots,N_1-1$ and hence $\sum_{j=1}^{N_1}(\cbe^{(1)}_{j,N_1})^\a\leq (1+\eta_1)^{N_1}$ whenever $\delta_1^{(1)}$ is sufficiently small and, in any case, satisfies
\begin{equation}\label{delta1}
\delta_1^{(1)}\in \Big(0,\Big(\dfrac{N_1-1}{(1+\eta_1)^{N_1}-\veps}\Big)^{1/\a}-y_1\Big).
\end{equation}
Moreover, thanks to  \eqref{betaLemma},
\[
\sum_{j\notin\CR_1}(\cbe^{(1)}_{j,N_1})^\a\geq\dfrac{|\{1,\dots,N_1\}\setminus \CR_1|}{(y_1+\delta_1^{(1)})^\a}=\dfrac{N_1-1}{(y_1+\delta_1^{(1)})^\a}
\]
whence
\[
\sum_{j\in\CR_1}(\cbe^{(1)}_{j,N_1})^\a=\Big(\sum_{j=1}^{N_1-1}-\sum_{j\notin\CR_1}\Big)(\cbe^{(1)}_{j,N_1})^\a\leq (1+\eta_1)^{N_1}-\dfrac{N_1-1}{(y_1+\delta_1^{(1)})^\a}<\veps
\]
where the last inequality follows from \eqref{delta1}. At this stage, one can say that if
\[
f_1(\cY)\in \{a_1\}\times I_1(\delta_1^{(1)})\times (\overline{\R}^2)^\infty
\]
then \eqref{betaLemma} and \eqref{sommaLemma} hold simultaneously with $n=1$. Now, to verify that the above assumption is non-empty, it is enough to use independence to prove that 
\[
\begin{split}
&\cP\Big(\bigcap_{k=1}^{N_1-1}\Big\{(\cL^{(1)}_k,\cR^{(1)}_k)\in C_k^{(1)}(\delta_1^{(1)})\times B_k^{(1)}(\delta_1^{(1)})\Big\}\Big)\\
&\qquad=\prod_{k=1}^{N_1-1}\cP\Big\{(\cL^{(1)}_k,\cR^{(1)}_k)\in C_k^{(1)}(\delta_1^{(1)})\times B_k^{(1)}(\delta_1^{(1)})\Big\}\\
&\qquad \geq \prod_{k=1}^{N_1-1}\cP\Big\{(\cL^{(1)}_k,\cR^{(1)}_k)\in B_{\rho_k}(x_k,y_k)\Big\}\\
&\qquad>0\qquad\text{(in view of \eqref{ip2})}
\end{split}
\]
holds whenever $B_{\rho_k}(x_k,y_k)$ is a suitable neighbourhood of radius $\rho_k$ and center $(x_k,y_k)$ in $\Gamma\cap (C_k^{(1)}(\delta_1^{(1)})\times B_k^{(1)}(\delta_1^{(1)}))$, with $\Gamma:=\{(x,y)\in\R^2:\;|x|^\a+|y|^\a=1\}$ and $k=1,\dots,N_1-1$. Coming back to the aim expressed at the beginning of this section, in view of the previous arguments, the first term of the sequence $(A_n)_{n\geq1}$ can be defined to be  
\[
A_1:=f^{-1}_1\Big(\{a_1\}\times I_1(\delta_1^{(1)})\times (\overline{\R}^2)^\infty\Big)
\]
which is a closed subset of $M$ since $f_1$ is continuous.\newline
\textbf{Step 4.} This step deals with the case of $n=2$. Starting from the same $N_1, N_2$, $\CR_1, \CR_2$, $a_1$ and $a_2=(\cnu^{(2)},\ci^{(2)})$ as in Step $1$, one repeats for $a_1$, seen as a subtree of $a_2$, the same intervals construction made in Step $2$ with a suitable $\delta_1^{(2)}$, to be determined at a latter time, in the place of $\delta_1^{(1)}$ in such a way that $\delta_1^{(2)}\leq \delta_1^{(1)}$. As a consequence, one has
\[
(\cL^{(2)}_k,\cR^{(2)}_k)_{k=1,\dots,N_1-1}\in I_1(\delta_1^{(2)}).
\]
Thus, in view of \eqref{betaLemma}, $\cR^{(2)}_{N_1}$ has to satisfy
\[
\cR^{(2)}_1 \cdot\cR^{(2)}_{N_1}\in\Big[\dfrac{1}{y_2},\dfrac{1}{y_2-\delta_2^{(2)}}\Big]
\]
for every value of $\cR^{(2)}_1$ in $B_1^{(1)}(\delta_1^{(2)})$, with a positive $\delta_2^{(2)}$ to be determined later. This condition is satisfied if
\[
\cR^{(2)}_{N_1}\in B_{N_1}^{(2)}(\delta_1^{(2)},\delta_2^{(2)}):= \Big[\dfrac{y_1+\delta_1^{(2)}}{y_2},\dfrac{y_1}{y_2-\delta_2^{(2)}}\Big]
\]
holds together with $\delta_1^{(2)}<\delta_2^{(2)}y_1/(y_2-\delta_2^{(2)})$. Like in Step $2$ one forces $\cL^{(2)}_{N_1}$ to satisfy
\[
\cL^{(2)}_{N_1}\in C_{N_1}^{(2)}(\delta_1^{(2)},\delta_2^{(2)}):=\text{range of }(1-x^\a)^{1/\a}\text{ as }x\text{ varies in } B_{N_1}^{(2)}(\delta_1^{(2)},\delta_2^{(2)}).
\]
One can proceed this way to obtain a family of intervals $\Big\{C_{k}^{(2)}(\delta_1^{(2)},\delta_2^{(2)}),$ $ B_{k}^{(2)}(\delta_1^{(2)},\delta_2^{(2)}):\;k=N_1,\dots,N_2-1\Big\}$ such that \eqref{betaLemma} is met for $n=2$ if
\[
(\cL^{(2)}_k,\cR^{(2)}_k)_{k=1,\dots,N_2-1}\in I_1(\delta_1^{(1)})\times I_2(\delta_1^{(2)},\delta_2^{(2)})
\]
with $I_2(\delta_1^{(2)},\delta_2^{(2)}):=\times_{k=N_1}^{N_2-1}\Big(C_{k}^{(2)}(\delta_1^{(2)},\delta_2^{(2)})\times B_{k}^{(2)}(\delta_1^{(2)},\delta_2^{(2)})\Big)$. Moreover, arguing as in the previous step, one sees that
\[
\begin{split}
&\cP\Big\{f_2(\cY)\in \{a_1\}\times\{a_2\}\times \Big(I_1(\delta_1^{(1)})\times (\overline{\R}^2)^\infty\Big)\\
&\qquad\qquad\qquad\qquad\qquad\qquad\qquad\qquad\times \Big( I_1(\delta_1^{(2)})\times I_2(\delta_1^{(2)},\delta_2^{(2)})\times (\overline{\R}^2)^\infty\Big)\Big\}
\end{split}
\]
is strictly positive and whence one can set $A_2:=f_2^{-1}\Big(\{a_1\}$ $\times\{a_2\}\times \Big(I_1(\delta_1^{(1)})\times (\overline{\R}^2)^\infty\Big)\times \Big( I_1(\delta_1^{(2)})\times I_2(\delta_1^{(2)},\delta_2^{(2)})\times (\overline{\R}^2)^\infty\Big)\Big)$. This way, $A_2$ turns out to be a closed subset of $A_1$.\newline
\textbf{Step 5.} This step extends the procedure to find $A_n$ for any $n$. Here, one confines oneself to analysing the case of odd $n$'s. In fact, when $n$ is even, the way of reasoning reduces to a simplified form of the odd case. Hence, let $m$ be an odd number and assume that \eqref{betaLemma} is satisfied for $n=m-1$ if
\begin{equation}\label{f_m-1}
\begin{split}
&f_{m-1}(\cY)\in\{a_1\}\times\dots\times \{a_{m-1}\}\times\Big[\times_{k=1}^{m-1}\Big(\times_{i=1}^k I_i(\delta_1^{(k)}),\dots,\delta_i^{(k)})\\
&\qquad\qquad\qquad\qquad\qquad\qquad\qquad\qquad\qquad\qquad\qquad\qquad\times (\overline{\R}^2)^\infty\Big)\Big].
\end{split}
\end{equation}
Starting from the same $N_1,\dots,N_m$, $\CR_1,\dots,\CR_m$, $a_1,\dots,a_{m-1}$ and $a_m=(\cnu^{(m)},\ci^{(m)})$ as in Step $1$, one replaces $\delta_1^{(m-1)},\dots,\delta_{m-1}^{(m-1)}$ with smaller $\delta_1^{(m)},\dots$ $\delta_{m-1}^{(m)}$ in such a way that, for a suitable $\delta_{m}^{(m)}>0$, one may determine intervals $C^{(m)}_{N_{m-1}}(\delta_1^{(m)},\dots,\delta_m^{(m)}),B^{(m)}_{N_{m-1}}(\delta_1^{(m)},$ $\dots,\delta_m^{(m)}),\dots,C^{(m)}_{N_{m}-1}(\delta_1^{(m)},\dots,\delta_m^{(m)}),$ \newline$B^{(m)}_{N_{m}-1}(\delta_1^{(m)},\dots,\delta_m^{(m)})$ for which \eqref{betaLemma} holds true also for $n=m$, whenever
\[
f_m(\cY)\in \{a_1\}\times \dots\times\{a_m\}\times\Big[\times_{k=1}^{m}\Big(\times_{i=1}^k I_i(\delta_1^{(k)}),\dots,\delta_i^{(k)})\times (\overline{\R}^2)^\infty\Big)\Big]
\]
with $I_m(\delta_1^{(m)},\dots,\delta_m^{(m)}):=\times_{k=N_{m-1}}^{N_m-1}\Big(C^{(m)}_{k}(\delta_1^{(m)},\dots,\delta_m^{(m)})\times B^{(m)}_{k}(\delta_1^{(m)},\dots,$ $\delta_m^{(m)})\Big)$. At this stage, $\delta_m^{(m)}$ has to be determined so that \eqref{sommaLemma} is met. From Step $1$, one has $1=\Big(\sum_{j\in\CR_m}+\sum_{j\notin\CR_m}\Big)\beta_{j,N_m}^\a<\veps+(N_m-N_{m-1})/y_m^\a$. Reasoning like in Step $2$, one considers positive numbers $h_m$ and $\eta_m$ satisfying
\[
h_m<\dfrac{N_m-N_{m-1}}{y_m^\a}-1+\veps\quad\text{and}\quad \eta_m<(1+h_m)^{1/N_m}-1
\]
and chooses 
\[
\delta_m^{(m)}<\Big(\dfrac{N_m-N_{m-1}}{(1+\eta_m)^{N_m}-\veps}\Big)^{1/\a}-y_m.
\]
One can get $(\cL^{(m)}_k)^\a+(\cR^{(m)}_k)^\a\leq 1+\eta_m$ for every $k=1,\dots,N_m-1$ by reducing $\delta_1^{(m)},\dots,\delta_m^{(m)}$ if needed, and then
\[
\sum_{j=1}^{N_m}(\cbe^{(m)}_{j,N_m})^\a\leq (1+\eta_m)^{N_m}.
\]
In view of \eqref{betaLemma}, 
\[
\sum_{j\notin\CR_m}(\cbe^{(m)}_{j,N_m})^\a\geq \dfrac{|\{1,\dots,N_m\}\setminus\CR_m|}{(y_m+\delta_m^{(m)})^\a}=\dfrac{N_m-N_{m-1}}{(y_m+\delta_m^{(m)})^\a}
\]
and hence, by definition of $\delta_m^{(m)}$,
\[
\sum_{j\in\CR_m}(\cbe^{(m)}_{j,N_m})^\a\leq (1+\eta_m)^{N_m}-\dfrac{N_m-N_{m-1}}{(y_m+\delta_m^{(m)})^\a}<\veps.
\]
Thus, \eqref{betaLemma}-\eqref{sommaLemma} hold for $n=m$ if
\[
f_m(\cY)\in \{a_1\}\times \dots\times\{a_m\}\times\Big[\times_{k=1}^{m}\Big(\times_{i=1}^k I_i(\delta_1^{(k)}),\dots,\delta_i^{(k)})\times (\overline{\R}^2)^\infty\Big)\Big].
\]
After noting that
\[
\cP\Big\{f_m(\cY)\in \{a_1\}\times \dots\times\{a_m\}\times\Big[\times_{k=1}^{m}\Big(\times_{i=1}^k I_i(\delta_1^{(k)},\dots,\delta_i^{(k)})\times (\overline{\R}^2)^\infty\Big)\Big]\Big\}>0
\]
one can choose
\[
A_n:=f_n^{-1} \Big(\{a_1\}\times \dots\times\{a_n\}\times\Big[\times_{k=1}^{n}\Big(\times_{i=1}^k I_i(\delta_1^{(k)}),\dots,\delta_i^{(k)})\times (\overline{\R}^2)^\infty\Big)\Big]\Big)
\]
which is a closed subset of $A_{n-1}$.\newline
\textbf{Conclusion.} The decreasing sequence $(A_n)_{n\geq1}$ constructed in the previous steps is formed of non-empty closed subsets of the compact set $M$. Hence, as granted by the finite intersection principle, $\bigcap_{n\geq1}A_n$ is non-empty, and the proof is completed by noting that $\cY^{-1}\Big(\bigcap_{n\geq1}A_n\Big)\neq\emptyset$ and that \eqref{betaLemma}-\eqref{sommaLemma} hold for every $\co_0$ in $\cY^{-1}\Big(\bigcap_{n\geq1}A_n\Big)$.

\section{Probability measures with symmetrized forms attracted by a stable law}\label{AppExample}
As told in the second last paragraph of Section \ref{sec:intro}, here is an example of p.d.f. which does not belong to the s.d.a. of any $\a$-stable law, whilst its symmetrized form does. Let $I$ and $S$ be two positive real numbers such that $I<S$ and let $c:=(S+I)/2$. One puts
\[
\begin{array}{ll}
G_I(x)&:=1-I x^{-\a}\I_{(1,+\infty)}(x)\\
G_S(x)&:=1-S x^{-\a}\I_{(1,+\infty)}(x)
\end{array}
\]
and chooses $s_1>1$. A continuous p.d.f. $F$ is now defined as follows. At first, one sets $F(s_1):=G_S(s_1)$. Then, one considers the derivative function $f$ of $F$ defined by $f(x)=k\a/x^{\a+1}$ for every $x$ in $(s_1,i_1)$, where $k$ is a fixed number in $(S,S+I)$ and $i_1$, greater than $s_1$, satisfies $F(i_1)=G_I(i_1)$. Let $s_2$ be the number, greater than $i_1$, which meets $F(i_1)=G_S(s_2)$, and let $f(x)=0$ on $(i_1,s_2)$. After $2(m-1)$ repetitions of the process, one gets the point $s_m$ and one defines the derivative on $(s_m,i_m)$ to be $f(x)=k\a/x^{\a+1}$, where $i_m>s_m$ satisfies $F(i_m)=G_I(i_m)$. In the next repetition, one sets $f(x)=0$ on $(i_m,s_{m+1})$, where $s_{m+1}$ $(>i_m)$ meets $F(i_m)=G_S(s_{m+1})$. This way, $F(x)$ is specified at every $x$ in $[s_1,+\infty)$ and $i_m^\a(1-F(i_m))=I$, $s_m^\a(1-F(s_m))=S$ for every $m\geq1$, so that
\begin{equation}\label{liminf<limsup}
\liminf_{x\to+\infty}x^\a(1-F(x))=I<S=\limsup_{x\to+\infty}x^\a(1-F(x)).
\end{equation}
Now, one extends $F$ to $(-\infty,-s_1]$ by setting $F(x):=2c |x|^{-\a}-1+F(-x)$ for every $x\leq -s_1$, and one completes the definition of $F$ by interpolating linearly on $(-s_1,s_1)$. The resulting function $F$ is a p.d.f. since the derivative of its restriction to $(-\infty,-s_1]$ is $(2c\a/(-x)^{\a+1}-f(-x))$ which is always positive. Indeed, by construction, $f(-x)\leq (S+I)\a/(-x)^{\a+1}=2c\a/(-x)^{\a+1}$ for every $x<-s_1$. On the one hand, gathering up all these remarks one can say that $F$ is a p.d.f. which, by virtue of \eqref{liminf<limsup}, cannot belong to the s.d.a. of any $\a$-stable distribution. On the other hand, the symmetrized form $F^*$ of $F$ satisfies
\[
x^\a(1-F^*(x))=2c \qquad \text{for every }x\in[s_1,+\infty)
\]
so that, on the difference of $F$, $F^*$ belongs to the s.d.a. of an $\a$-stable distribution.
}




\end{document}